\documentclass[12pt, reqno]{amsart}
\usepackage{amssymb, amsthm, amsmath, amsfonts}
\usepackage{array, epsfig}
\usepackage{bbm, enumerate, textcomp}
\usepackage{color}
\usepackage{hyperref,verbatim}
\usepackage{lineno}

  \evensidemargin
\oddsidemargin
\newcommand*\patchAmsMathEnvironmentForLineno[1]{%
  \expandafter\let\csname old#1\expandafter\endcsname\csname #1\endcsname
  \expandafter\let\csname oldend#1\expandafter\endcsname\csname end#1\endcsname
  \renewenvironment{#1}%
     {\linenomath\csname old#1\endcsname}%
     {\csname oldend#1\endcsname\endlinenomath}}%
\newcommand*\patchBothAmsMathEnvironmentsForLineno[1]{%
  \patchAmsMathEnvironmentForLineno{#1}%
  \patchAmsMathEnvironmentForLineno{#1*}}%
\AtBeginDocument{%
\patchBothAmsMathEnvironmentsForLineno{equation}%
\patchBothAmsMathEnvironmentsForLineno{align}%
\patchBothAmsMathEnvironmentsForLineno{flalign}%
\patchBothAmsMathEnvironmentsForLineno{alignat}%
\patchBothAmsMathEnvironmentsForLineno{gather}%
\patchBothAmsMathEnvironmentsForLineno{multline}%
}

\begin{document}

\newcommand{\be}{\begin{equation}}
\newcommand{\ee}{\end{equation}}
\newcommand{\bea}{\begin{eqnarray}}
\newcommand{\eea}{\end{eqnarray}}
\newcommand{\beaa}{\begin{eqnarray*}}
\newcommand{\eeaa}{\end{eqnarray*}}
\newcommand{\Var}{\mathop{\mathrm{Var}}\nolimits}

\renewcommand{\proofname}{\bf Proof}
\newtheorem{conj}{Conjecture}
\newtheorem*{cor*}{Corollary}
\newtheorem{cor}{Corollary}[section]
\newtheorem{proposition}{Proposition}[section]
\newtheorem{lemma}{Lemma}[section]

\newtheorem{taggedlemmax}{Lemma}
\newenvironment{taggedlemma}[1]
 {\renewcommand\thetaggedlemmax{#1}\taggedlemmax}
 {\endtaggedlemmax}
 
\newtheorem*{crit}{Conditions for conservativity}
\newtheorem*{Hopf}{Hopf's ratio ergodic theorem}
\newtheorem{theorem}{Theorem}[section]

\theoremstyle{remark}
\newtheorem{remark}{Remark}[section]

\newfont{\zapf}{pzcmi}

\def\PP{\mathrm{P}}
\def\EE{\mathrm{E}}

\def\R{\mathbb{R}}
\def\Q{\mathbb{Q}}
\def\Z{\mathbb{Z}}
\def\N{\mathbb{N}}
\def\E{\mathbb{E}}
\def\P{\mathbb{P}}
\def\V{\mathbb{D}}
\def\ZZ{\mathcal{Z}}
\def\XX{\mathcal{X}}
\def\BB{\mathcal{B}}
\def\FF{\mathcal{F}}
\def\I{\mathbbm{1}}
\newcommand{\D}{\hbox{\zapf D}}
\newcommand{\eps}{\varepsilon}
\newcommand{\sgn}{\mathop{\mathrm{sign}}\nolimits}
\newcommand{\eqdistr}{\stackrel{d}{=}}
\newcommand{\supp}{\mathop{\mathrm{supp}}\nolimits}
\newcommand{\lin}{\mathop{\mathrm{lin}}\nolimits}
\newcommand{\Int}{\mathop{\mathrm{Int}}\nolimits}
\newcommand{\Cl}{\mathop{\mathrm{Cl}}\nolimits}
\newcommand{\Inv}{\mathop{\mathrm{Inv}}\nolimits}
\newcommand{\Sym}{\text{\rm Sym}}
\newcommand{\dist}{\operatorname{dist}}
\newcommand{\tc}{\textcolor{red}}
\renewcommand{\Re}{\operatorname{Re}}
\renewcommand{\mod}{\operatorname{mod}}
\newcommand{\Mod}[1]{\ \mathrm{mod}\ #1}

\newcommand{\overbar}[1]{\mkern 1.5mu\overline{\mkern-3mu#1\mkern-3mu}\mkern 1.5mu}
\newcommand{\todistr}{\overset{d}{\underset{n\to\infty}\longrightarrow}}
\newcommand{\toprobab}{\overset{\P}{\underset{n\to\infty}\longrightarrow}}

\title[Stationary entrance chains and applications to random walks]{Stationary entrance chains \\and applications to random walks}
\author{Aleksandar Mijatovi\'c} 
\address{Aleksandar Mijatovi\'c, Department of Statistics, University of Warwick \& The Alan Turing Institute, UK}
\email{a.mijatovic@warwick.ac.uk}

\author{Vladislav Vysotsky}
\address{Vladislav Vysotsky, Department of Mathematics, University of Sussex,
Brighton BN1 9QH,
United Kingdom}
\email{v.vysotskiy@sussex.ac.uk}


\begin{abstract}
For a Markov chain $Y$ with values in a Polish space, consider the {\it entrance chain} obtained by sampling $Y$ at the moments when it enters a fixed set $A$ from its complement $A^c$. Similarly, consider the {\it exit chain}, obtained by sampling $Y$ at the exit times from $A^c$ to $A$. We use the method of inducing from ergodic theory to study invariant measures of these two types of Markov chains in the case when the initial chain $Y$ has a known invariant measure. We give explicit formulas for invariant measures of the entrance and exit chains under certain recurrence-type assumptions on $A$ and $A^c$, which apply even for transient chains. Then we study uniqueness and ergodicity  of these invariant measures assuming that $Y$ is topologically recurrent, topologically irreducible, and weak Feller. 

We give applications to random walks in $\R^d$, which we regard as
``stationary''  Markov chains started under the Lebesgue measure. We are mostly interested in dimension one, where we study the Markov chain of overshoots above the zero level of a random walk that oscillates between $-\infty$ and $+\infty$. We show that this chain is ergodic, and use this result to prove a central limit theorem for the number of level crossings of a random walk with zero mean and finite variance of increments.  
\end{abstract}

\maketitle

\section{Introduction} \label{sec:Intro} Let $S= (S_n)_{n \ge 0}$ be a non-degenerate random walk in $\R^d$, where $d \in \N$. That is $S_n= S_0 + X_1+ \ldots + X_n$ for $n \in \N$, where  $X_1, X_2, \ldots$ are the independent identically distributed increments and $S_0$ is the starting point that is independent of the increments.  

For now consider the case $d=1$ and assume that the random walk $S$ {\it oscillates}, that is $\limsup S_n = -\liminf S_n = +\infty$ a.s.\ as $n \to \infty$. Then either $\E X_1=0$ or $\E X_1$ does not exist; in particular, in the latter case $S$ can be transient. Define the {\it crossing times} of the zero level by $\mathcal{T}_0:=0$ and
\[
\mathcal{T}_{n}:= \inf\{k>\mathcal{T}_{n-1}: S_{k-1} < 0, S_k \ge 0 \text{ or } S_{k-1} \ge 0, S_k <0\}, \qquad n \in\N,
\]
and let 
\begin{equation} \label{eq: crossing chain}
\mathcal{O}_n:=S_{\mathcal{T}_n}, \qquad n \in \N, 
\end{equation}
be the corresponding {\it overshoots}. It is easy to show, using that the $\mathcal{T}_n$'s are stopping times, that the sequence $\mathcal{O}:=(\mathcal{O}_n)_{n \ge 1}$ is a Markov chain. 

This paper was motivated by our interest in stationarity and stability properties of the Markov chain of overshoots $\mathcal{O}$. In~\cite{MijatovicVysotsky} we essentially showed that the measure
\begin{equation} \label{eq: pi}
\pi(dx):= \bigl[ \I_{[0, \infty)}(x) \P(X_1 > x) + \I_{(-\infty,0)}(x) \P(X_1 \le x) \bigr] \lambda(dx), \qquad x \in \ZZ,
\end{equation}
is invariant for this chain, where $\ZZ$ denotes the minimal topologically closed subgroup of $(\R, +)$ that contains the topological support of the distribution of $X_1$ and $\lambda$ denotes the normalized Haar measure on $(\ZZ, +)$. To clarify, either $\ZZ=\R$ and $\lambda$ is the Lebesgue measure, or $\ZZ = h \Z$ and $\lambda$ is the counting measure multiplied by $h$ for some $h>0$. Note that $\pi$ is finite if and only if $\E |X_1| < \infty$, in which case  $\E X_1 =0$ by the assumption of oscillation.  

In~\cite{MijatovicVysotsky} we first computed $\pi$ heuristically for a specific type of random walks using ergodic averaging argument. Then we proved the invariance of $\pi$ for the overshoots of general random walks using a quite complicated ad-hoc argument based on time-reversibility; this argument is clarified and generalized in this paper. The same approach of deriving or even guessing an invariant measure and then checking its invariance was used in a number of other works  concerning stability of certain related Markov chains, e.g.\ in~\cite{Borovkov, Knight, PeigneWoessUnpub}. In all these examples such method neither explains the form of the invariant measure found nor shows how to proceed in a different setting. Moreover, uniqueness of the invariant measure (up to a constant factor) shall be established separately -- for example, in~\cite[Section 3]{MijatovicVysotsky} we did this only when $\E X_1=0$ and under the additional assumptions on the distribution of $X_1$, which ensured convergence $\P(\mathcal{O}_n \in \cdot) \to \pi/\pi(\ZZ)$  in the total variation distance as $n \to \infty$ for every starting point $S_0=x$. 

This paper presents a unified approach to {\it finding}
invariant measures and proving their uniqueness and ergodicity. It applies in a much more general context than level-crossings of one-dimensional random walks. Our method is built on {\it inducing}, a basic tool of ergodic theory, introduced by S.~Kakutani in 1943. In order to proceed to a general setting, note that the chain of overshoots has a periodic structure since its values at consecutive steps have different signs. Therefore, it suffices to consider the non-negative Markov chain $O=(O_n)_{n\ge 1}$ of {\it overshoots at up-crossings} defined by $O_n:= \mathcal{O}_{2n - \I(S_0<0)}$. We will also consider the sequence $U=(U_n)_{n \ge 1}$ of {\it undershoots at up-crossings} given by $U_n:= \mathcal{U}_{2n - \I(S_0<0)}$, where $ \mathcal{U}_n:=S_{\mathcal{T}_n-1}$. This latter sequence turns out to be a Markov chain, but this fact is far less intuitive since $\mathcal{T}_n-1$ are not stopping times. The chain $U$ played an important role in the proof of invariance of $\pi$ presented in~\cite{MijatovicVysotsky}.

Observe that the Markov chain of overshoots $O$ at up-crossings above the zero level is obtained by
sampling the one-dimensional  random walk $S$ at the moments it enters the set $[0,
\infty)$ from $(-\infty, 0)$. 
Similarly, for any Markov chain\footnote{All Markov chains considered in this paper are time-homogeneous.}  $Y$ with values in a Polish space $\XX$, we can consider the {\it entrance Markov chain}, denoted by $Y^{\rangle A}$, that is obtained by sampling $Y$ at the moments of  entry into an arbitrary fixed Borel set $A$ from its complement $A^c$. We also consider the {\it exit Markov chain}, denoted by $Y^{A^c \rangle}$, obtained by sampling $Y$ at the exit times 
from $A^c$ to $A$; the Markov property of this sequence is not obvious and we refer to Lemma~\ref{lem: Markov} for its proof. In this notation, we have $O=S^{\rangle A}$ and $U=S^{A^c \rangle}$ for $A= [0, \infty) \cap \ZZ$. 


We will show (Theorems~\ref{thm: inducing MC recurrent} and \ref{thm: invariant dual}) that if $Y$ has an  invariant $\sigma$-finite measure~$\mu$, then the entrance chain $Y^{\rangle A}$ and the exit chain $Y^{A^c \rangle}$ have the respective invariant  measures 
\begin{equation} \label{eq: measures def}
\mu_A^{entr}(dx)=\P_x(\hat{Y}_1 \in A^c) \mu(dx) \text{ on }A \quad \text{and} \quad 
\mu_{A^c}^{exit}(dx)= \P_x(Y_1 \in A) \mu(dx) \text{ on } A^c,
\end{equation}
where $\hat Y$ is a Markov chain that is {\it dual} to $Y$ relative to $\mu$ and satisfies $\P_x(Y_0 = x)=1=\P_x(\hat Y_0 = x)$ for every $x \in \XX$, provided  that 
$Y$ and $\hat Y$ visit both sets $A$ and $A^c$ infinitely often $\P_x$-a.s.\ for $\mu_{A^c}^{exit}$-a.e.~$x$ and for $\mu_A^{entr}$-a.e.~$x$.
In particular, these assumptions are satisfied if $Y$ is recurrent starting under $\mu$ and $\mu_{A^c}^{exit}(A^c)>0$, that is, $Y$  can get from $A^c$ to $A$. In this case the chains $Y^{\rangle A}$ and $Y^{A^c \rangle}$ are recurrent and also ergodic if  so is $Y$ started under $\mu$.  However, we stress that our results also apply when $Y$ is {\it transient}.

Note that formulas~\eqref{eq: measures def} are symmetric in the sense that their right-hand sides interchange if we swap the chain $Y$ and the set $A$ with the dual chain $\hat Y$ and the complement set $A^c$. The reason is that the exit chain $Y^{A^c \rangle}$ of $Y$ from $A^c$ to $A$ turns out to be dual to the entrance chain $\hat Y^{\rangle A^c}$ of  $\hat Y$ into $A^c$ from $A$  relative to the measure $\mu_{A^c}^{exit}$ (Proposition~\ref{prop: duality}). This immediately implies that  $\mu_{A^c}^{exit}$  is invariant for the exit chain  $Y^{A^c \rangle}$. This in turn yields the invariance of $\mu_A^{entr}$ for the entrance chain $Y^{\rangle A}$ by swapping $Y$ and $A$ with $\hat Y$ and $A^c$. The described duality between the entrance and the exit chains explains the need to consider the latter ones. 

Our further result concerns ergodicity and uniqueness of the invariant measures for the entrance and exit chains. In Theorem~\ref{thm: inducing bijection} we show that under the topological assumptions of recurrence, irreducibility, and the weak Feller property of the chain $Y$ and, essentially, non-emptiness of the interiors of the sets $A$ and $A^c$, the questions of existence of an invariant measure, its ergodicity and uniqueness (up to a constant factor)  in the class of locally finite Borel measures have the same answer simultaneously for each of the three chains $Y$, $Y^{\rangle A}$, $Y^{A^c \rangle}$. 

Next we consider applications of the general results described above to random walks on $\R^d$ (see Section~\ref{Sec: Application to RWs}). In this case $\lambda$, the normalized Haar measure on the minimal closed subgroup $\ZZ$ of $(\R^d,+)$ that contains the support of $X_1$, is invariant for the random walk $S$. This explains why we need to use the results of infinite ergodic theory. Since the dual of $S$ relative to $\lambda$ is $-S$, the first formula in \eqref{eq: measures def} reads
\[
\lambda_A^{entr}(dx)=\P(X_1 \in x - A^c) \lambda(dx), \qquad x \in A. 
\]
If the random walk $S$ is topologically recurrent on $\ZZ$, it is know that the Haar measure $\lambda$ is the unique locally finite invariant measure of $S$, and $\lambda$ is ergodic for $S$. Then the measure $\lambda_A^{entr}$ is uniquely invariant and ergodic for the entrance chain $S^{\rangle A}$ when $\lambda(A)>0$, $\lambda(A^c)>0$, and $\lambda(\partial A)=0$ (Theorem~\ref{thm: RW general}).


To give a concrete example, consider the orthant $A=\{x \in \R^d: x \ge 0\}$, where and below the inequalities between points in $\R^d$ are understood coordinate-wise. Assume that $S$ hits the interiors of $A$ and $-A$ a.s.\ when starting at $S_0=0$; in dimension one this assumption is equivalent to oscillation of $S$. Then $\lambda_A^{entr}$ can be written as
\begin{equation} \label{eq: inv quadrant}
\pi_+(dx):=  (1 - \P(X_1 \le x)) \lambda(dx), \qquad x \in \ZZ \cap [0, \infty)^d.
\end{equation}
In particular, for $d=1$ this means that $\pi_+$ is invariant for the chain $O$. Combining this with an analogous result for the chain of overshoots at down-crossings of zero yields the stated invariance of the measure $\pi$ for the chain $\mathcal O$ (Corollary~\ref{cor: orthant}). This invariant measure is unique and ergodic when $S$ is topologically recurrent (in particular, this settles the question of uniqueness of $\pi$  in dimension $d=1$, only partially answered in~\cite{MijatovicVysotsky}). 


Our interest in stationarity of overshoots and level-crossings of one-dimensional random walks was motivated as follows. First, the overshoots are related to the {\it local times} of random walks. Perkins~\cite{Perkins} defined\footnote{There is no canonical definition of local times of random walks, see Cs\"org\H o and R\'ev\'esz~\cite{CsorgoRevesz} and Mijatovi\'c and Uribe Bravo~\cite{MijatovicUribe-Bravo} for other versions.} the local time of $S$ at zero (at time $n$) as $\sum_{k=1}^{L_n} |\mathcal{O}_k|$, where
\begin{equation} \label{eq: Ln def}
L_n:=\max\{k \ge 0 : \mathcal{T}_k \le n\}
\end{equation}
denotes the {\it number of zero-level crossings} of the walk by time $n$. Then~\cite{Perkins} proved a limit theorem for the local time, assuming that the walk has zero mean and finite variance. From this result and ergodicity of the Markov chain $\mathcal{O}$ (established in Theorem~\ref{thm: uniqueness}), we obtain a limit theorem for the number of level crossings $L_n$ (Theorem~\ref{thm: level-crossings}).

Second, the chain $O$ appeared in the study of the probabilities that the {\it integrated random walk} $(S_1 + \ldots + S_n)_{n \ge 1}$ stays positive for a long time; see Vysotsky~\cite{Vysotsky2010, Vysotsky2014}. The main idea of the approach of~\cite{Vysotsky2010, Vysotsky2014} is in a) splitting the trajectory of the walk into consecutive ``cycles'' between the up-crossing times; and b)  using that for certain distributions of increments, e.g.\ when the distribution $\P(X_1 \in \cdot \, | X_1>0)$ is exponential, the overshoots $(O_n)_{n \ge 1}$ are  i.i.d.\ regardless of the starting point $S_0$. This paper was originally motivated by the question whether this approach can be extended to general distributions of increments if $S$ is started such that $O$ remains stationary. 

Third, the level-crossings define the dynamics of the so-called {\it switching random walks}. This a special type of Markov chains with the transition probabilities  of the form $P(x, dy) = P_{\sgn x}(dy -x)$ for $x \neq 0$ and $P(0, dy) = \alpha P_+(dy) + (1-\alpha) P_-(dy)$, where $P_+$ and $P_-$ are two probability distributions on  $\R$ and $\alpha \in [0,1]$. Such chains, introduced by Kemperman~\cite{Kemperman} under the name of {\it oscillating random walks}.
In the antisymmetric case
$P_+(dy)=P_-(-dy)$, the absolute values of the terms of such chain form the other Markov chain, called a {\it reflected random walk}. The chains of this type received a lot of attention, see Peign\'e and Woess~\cite{PeigneWoess} for references and generalizations. Invariant distributions for the reflected and switching random walks are known in some cases, see Borovkov~\cite{Borovkov} and~\cite{PeigneWoess}. We will generalize these results and clarify  connections to the classical stationary distributions\footnote{These distributions have the same form as $\pi_+$  in $d=1$, as discussed in~\cite[Sections 2.1 and~2.2]{MijatovicVysotsky}.} of the renewal theory in a separate paper~\cite{Vysotsky2024}, which uses the ideas developed here. 

We are not aware of any works concerning the entrance and exit Markov chains in any generality. We are also not aware of any applications of inducing in the problems related to level-crossings of one-dimensional random walks, and the idea to regard them as ``stationary'' processes starting from the Haar measure is new in this context. Here the classical and universal  tool is the Wiener--Hopf factorization, which does not yield much for our problem. In particular, this factorization was used in the works by Baxter~\cite{Baxter},
Borovkov~\cite{Borovkov}, and Kemperman~\cite{Kemperman}, which are   closely related to the questions considered in our paper. As for higher dimensional generalizations, we believe that our formulas for the invariant measures, such as \eqref{eq: inv quadrant}, are the only explicit results available. 

Finally, let us briefly describe possible applications of the general results of Sections~\ref{sec: stationarity MC} and~\ref{sec: duality} to reversible Markov chains, which is a wide class of chains with a known invariant distribution. For such chains, formula \eqref{eq: measures def} for $\mu_A^{entr}$ is particularly simple since we can take $\hat Y_1 =Y_1$. To verify the recurrence assumptions of Theorems~\ref{thm: inducing MC recurrent} and~\ref{thm: invariant dual}, one can use the following results. A simple necessary and sufficient condition for recurrence of countable reversible Markov chains is due to Lyons~\cite{Lyons}; further criteria for recurrence of general chains are given in Menshikov~et~al.~\cite[Section~2.5]{Menshikov+}. Conditions for recurrence of a set for general transient Markov chains on a countable state space can be found e.g.\ in Bucy~\cite{Bucy} and Murdoch~\cite{Murdoch}. For transient chains, there is one example with a particularly simple characterization of recurrent sets: by Gantert et al.~\cite[Theorem~1.7]{GPV}, a planar simple random walk conditioned on never hitting the origin visits any infinite subset of $\Z^2$ infinitely often a.s.

\subsection{Structure of the paper} 
In Section~\ref{Sec: setup} we carefully define the entrance and exit sequences sampled from a Markov chain and prove their Markov property. In Section~\ref{sec: stationarity MC} we study stationarity of these chains using the idea of inducing from ergodic theory -- in Section~\ref{sec: setup MCs} we provide a self-contained setup needed to apply inducing in the context of Markov chains; in Section~\ref{sec: Inducing MCs} we show the use of inducing in finding invariant measures for the entrance and exit chains sampled from a general recurrent Markov chain; and in Section~\ref{sec: MC e and u} we study existence and uniqueness of these invariant measures for the specific class of recurrent  weak Feller chains on metric spaces. In Section~\ref{sec: duality} drop the assumption of recurrence and 
explore the duality between the entrance and exit chains and its role in proving invariance of the measures defined in~\eqref{eq: measures def}. The rest of the paper concerns applications of the general results of Sections~\ref{sec: stationarity MC} and~\ref{sec: duality} to random walks. In Section~\ref{Sec: Application to RWs} we study the entrance chains sampled from random walks in $\R^d$, including the chains of overshoots in dimension one. In Section~\ref{Sec: L_n} we prove a limit theorem for the number of level-crossings. The Appendix contains some relevant basic facts from the infinite ergodic theory. 

\section{Entrance and exit Markov chains} \label{Sec: setup} 

In this section we set up the basic notation, and show that the subsequence obtained by sampling a Markov chain at the exit times from a set is again a Markov chain. 

Throughout this paper $(\XX, \FF)$ will be a measurable space. For a measure $\mu$ on $(\XX, \FF)$ and a non-empty set $A \in \FF$, by $\mu_A$ we denote the measure on $(A, \FF_A)$ given by $\mu_A:={\mu|}_{\FF_A}$, where $\FF_A:=\{B \subset A: B \in \FF\}$. If $\XX$ is a metric space, we always equip it with the Borel $\sigma$-algebra $\mathcal{B}(\XX)$ and refer to measures on $(\XX, \BB(\XX))$ as Borel measures on $\XX$ (for example, in this case $\mu_A$ is a Borel measure on $A$).

Throughout this paper $Y=(Y_n)_{n \ge 0}$ will be a time-homogeneous Markov chain taking values in $\XX$. By saying this, we assume that $Y$ is defined on some generic probability space $(\Omega, \mathcal A, \P)$ and $Y$ has a probability transition kernel $P$ on $(\XX, \FF)$ under $\P$. To simplify the notation, it is convenient to assume that $(\Omega, \mathcal A)$ is also equipped with a family of probability measures $\{\P_x \}_{x \in \XX}$ such that: $Y$ is a Markov chain with the transition kernel $P$ under $\P_x$ and $\P_x(Y_0=x) =1$ for every $x \in \XX$, and the function $x \mapsto \P_x(Y \in B)$ is measurable for any set $B \in \FF^{\otimes \N_0}$, where $\N_0:=\N \cup \{0\}$. Such a family of measures always exists for any probability kernel on $\XX$ when $(\Omega, \mathcal A)$ is the canonical space $(\XX^{\N_0}, \FF^{\otimes \N_0})$ and $Y$ is its identity mapping by the Ionescu Tulcea extension theorem (Kallenberg~\cite[Theorem~6.17]{Kallenberg}).  

Furthermore, for any measure $\nu$ on $(A, \FF_A)$, where $A \in \FF$ is non-empty, denote $\P_\nu := \int_A \P_x(\cdot) \nu(dx)$. Then $Y_0$ has  ``distribution'' $\nu$ under $\P_\nu$, in which case we say that $Y$ {\it starts under} $\nu$. Although $\nu$ is not necessarily a probability, we prefer to (ab)use probabilistic notation and terminology as above, instead of using respective notions of general measure theory. Denote by $\E_x$ and $\E_\nu$ respective expectations (Lebesgue integrals) over $\P_x$ and~$\P_\nu$. We say that a measure $\nu$ on $(A, \FF_A)$, where $A \in \FF$, is {\it invariant} for $Y$ if $\int_A P(x, B) \nu(dx) =\nu(B)$ for every $B \in \FF_A$. 

Define the entrance times of $Y$ to a set $A \in \XX$ from $A^c$ by $T_0^{\rangle A}:=0$ and
\[  
T_{n}^{\rangle A}:= \inf\{k>T_{n-1}^{\rangle A}: Y_{k-1}\in A^c, Y_k \in A \} 
\]
for $n \in \N$, where $\inf_\varnothing:=\infty$ by convention. The respective positions of $Y$ when entering $A$ from  $A^c$ and exiting from $A^c$ to $A$ are denoted by
\[ 
Y_n^{\rangle A}:=Y_{T_{n}^{\rangle A}} \quad \text{and} \quad Y_n^{A^c \rangle}:=Y_{T_{n}^{\rangle A}-1}
\]
for $n \in \N$, where we put $Y_\infty:=\dagger$ and denote by $\dagger $ the ``cemetery'' state, that is an additional point that does not belong to $\XX$. These variables are random elements of $(A_\dagger, \FF_A^\dagger)$ and $(A^c_\dagger, \FF_{A^c}^\dagger)$, respectively, where for any $B \in \FF$ we define 
\[
B_\dagger := B \cup \{\dagger\} \quad \text{ and } \quad \FF_B^\dagger:=\FF_B \cup \{C \cup \{\dagger\}: C \in \FF_B\},
\]
and write $A^c_\dagger $ for ${(A^c)}_\dagger$. Put $Y^{\rangle A}:={(Y^{\rangle A}_n)}_{n \ge 1}$ and $Y^{A^c \rangle}:= {(Y^{A^c \rangle}_n)}_{n \ge 1}$.

To identify for which initial values of $Y$ all entrance times $T_n^{\rangle A}$ are finite, put
\begin{equation} \label{eq: N}
N_A(Y):= \{x \in \XX: \P_x(Y_k \in A \text{ i.o.}, Y_k \in A^c \text{ i.o.})=1\},
\end{equation}
where ``i.o.'' stands for ``infinitely often''; we will write $N_A$ in short when the reference to $Y$ is unambiguous. This set is measurable. It is absorbing for $Y$, in the sense that
\begin{equation} \label{eq: NA}
\P_x(Y_1 \in N_A)= 1, \qquad x \in N_A.
\end{equation}
Indeed,  for every $x \in N_A$ we have
\[
1= \P_x( Y_k \in A \text{ i.o.}, Y_k \in A^c \text{ i.o.})=\int_\XX \P_x(Y_k \in A \text{ i.o.}, Y_k \in A^c \text{ i.o.} | Y_1 =y) \P_x(Y_1 \in dy). 
\]
Hence $\P_y(Y_k \in A \text{ i.o.}, Y_k \in A^c \text{ i.o.})=1$, that is $y \in N_A$ for $\P_x(Y_1 \in \cdot) $-a.e.\ $y$. This proves~\eqref{eq: NA}.

Furthermore, define the {\it exit sets} for $Y$:
\begin{equation} \label{eq: Aex}
B_{ex}(Y):=\{x \in B: \P_x(Y_1 \not \in B)>0\}, \qquad B\in \FF.
\end{equation}
These sets  are measurable. We will write $A^c_{ex}(Y)$ or simply $A^c_{ex}$ instead of ${(A^c)}_{ex}(Y)$. 

We will refer to the sequences  $Y^{\rangle A}$ and $Y^{A^c \rangle}$ respectively as the {\it entrance} and {\it exit} Markov chains. This is justified by the following result.

\begin{lemma} \label{lem: Markov}
Let $Y$ be a Markov chain that take values in a measurable space $(\XX, \FF)$, and let $A  \in \FF$. Then for every $x_0 \in \XX$, the entrance sequence $Y^{\rangle A}$ and the exit sequence $Y^{A^c \rangle}$ are time-homogeneous Markov chains under $\P_{x_0}$ and their transition probabilities are given by
\begin{equation} \label{eq: entr kernel}
P_A^{entr}(x, dy):= \P_x(Y_1^{\rangle A} \in dy), \qquad x\in A, y \in A \cup \{ \dagger \},
\end{equation}
\begin{equation} \label{eq: exit chain}
P_{A^c}^{exit}(x, dy):= \int_A  \P_z( Y_1^{A^c \rangle} \in dy) \P_x(Y_1 \in dz|Y_1 \in A), \quad x \in A^c_{ex}, y \in A^c_{ex} \cup \{ \dagger\},
\end{equation}
and 
\[
P_{A^c}^{exit}(\dagger, \{\dagger\}) :=1, \quad P_A^{entr}(\dagger, \{\dagger\}):=1.
\]
If $x_0 \in N_A$, these chains take values in the sets $A \cap N_A$ and $A^c_{ex} \cap N_A$, respectively.  
\end{lemma}

It is convenient to extend the transition kernel of the exit chain to the whole of $A^c$, say, by putting $P_{A^c}^{exit}(x, \{\dagger\}) := 1$ for $x \in A^c \setminus A^c_{ex}$. Thus, we can formally regard $Y^{A^c \rangle}$ as a Markov chain either on $A^c_\dagger$ or on $A^c_{ex} \cap N_A$, and we regard $Y^{\rangle A}$ as a Markov chain either on  $A_\dagger$ or on $A \cap N_A$. We say that a measure $\nu$ on $(A, \FF_A)$ is {\it proper} for the entrance chain $Y^{\rangle A}$  if $\nu(A\setminus N_A) = 0$. Similarly, a measure $\nu$ on $(A^c, \FF_{A^c})$ is {\it proper} for  the exit chain $Y^{A^c \rangle}$ if $\nu( A^c \setminus  (A^c_{ex} \cap N_A)) = 0 $. We will be interested only in proper invariant measures of these chains. 

\begin{proof} 
It is clear that $P_A^{entr}$ is a probability kernel on $(A_\dagger, \FF_A^\dagger)$. Its restriction to $(A \cap N_A, \FF_{A \cap N_A})$ is also a probability kernel since $\P_x(T^{\rangle A}_1<\infty)=1$ for every $x \in N_A$ and the set $N_A$ is absorbing for $Y$ by~\eqref{eq: NA}. Similarly, $P_{A^c}^{exit}$ is  a probability kernel on $(A^c_\dagger, \FF_{A^c}^\dagger)$ and its restriction to $(A^c_{ex} \cap N_A, \FF_{A^c_{ex} \cap N_A})$ is a probability kernel too. This implies the last claim of the lemma. 

Fix an $x_0 \in \XX$. Since the entrance times $T^{\rangle A}_n$ are increasing stopping times with respect to $Y$, it follows from the equality $Y^{\rangle A}=Y_{T^{\rangle A}}$ and a standard argument based on the strong Markov property of $Y$ under $\P_{x_0}$ that $Y^{\rangle A}$ is an $A_\dagger$-valued Markov chain under $\P_{x_0}$. The formula for its transition kernel is evident. However, the Markov property of the exit sequence $Y^{A^c \rangle}$ is not evident since  $(T_n^{\rangle A}-1)_{n \ge 1}$ are not stopping times. 

To prove that  $Y^{A^c \rangle}$ is Markov chain under $\P_{x_0}$ with values in  $A^c_{ex} \cup \{\dagger\}$ and the transition kernel $P_{A^c}^{exit}$ (in short, $P_{A^c}^{ex}$), it suffices to show that for any integer $n \ge 2$ and measurable sets $B_1, B_2, \ldots \subset A^c_{ex} \cup \{\dagger \}$, 
\[
\P_{x_0}(Y_1^{A^c \rangle} \in B_1, \ldots, Y_n^{A^c \rangle} \in B_n) = \int_{B_1} \P_{x_0}(Y_1^{A^c \rangle} \in dx_1) \int_{B_2} P_{A^c}^{ex}(x_1, d x_2) \ldots  \int_{B_n} P_{A^c}^{ex}(x_{n-1}, d x_n). 
\]

The proof is by induction. Denote $\bar B_1 := B_1  \setminus \{\dagger\}$. Let $n=2$, then
\begin{align*}
\P_{x_0}(Y_1^{A^c \rangle} \in \bar B_1, Y_2^{A^c \rangle} \in B_2) &=\sum_{k=1}^\infty \P_{x_0}(T_1^{\rangle A}=k, Y_{k-1} \in \bar B_1, Y_2^{A^c \rangle} \in B_2) \\
&= \sum_{k=1}^\infty \P_{x_0}\big(T_1^{\rangle A}>k-1, Y_{k-1} \in \bar B_1, Y_k \in A, Y_2^{A^c \rangle} \in B_2\big) \\
&=\sum_{k=1}^\infty \int_{\bar B_1} \P_{x_0}(T_1^{\rangle A}>k-1, Y_{k-1} \in d x_1) \\ 
&\qquad \qquad \times \P_{x_0} \big( Y_k \in A, Y_2^{A^c \rangle} \in B_2\big| \big.Y_{k-1} = x_1, T_1^{\rangle A}>k-1 \big).
\end{align*}
By the Markov property of $Y$, for $\P_{x_0}(Y_{k-1} \in \cdot)$-a.e.\ $x_1 \in \bar B_1$ and every $k \ge 1$ it is true that
\begin{align*}
\P_{x_0}\big( Y_k \in A, Y_2^{A^c \rangle} \in B_2\big| \big.Y_{k-1} = x_1, T_1^{\rangle A}>k-1 \big) &= \P_{x_1}( Y_1 \in A, Y_2^{A^c \rangle} \in B_2) \\
&= \int_A \P_z (Y_1^{A^c \rangle} \in B_2) \P_{x_1}(Y_1 \in d z). 
\end{align*}

On the other hand, from definition~\eqref{eq: exit chain} of  $P_{A^c}^{ex}$ we see that for every $x_1 \in \bar B_1$,
\begin{equation} \label{eq: P1}
\int_A \P_z (Y_1^{A^c \rangle} \in B_2) \P_{x_1}(Y_1 \in d z) = \P_{x_1}(Y_1 \in A) P_{A^c}^{ex}(x_1, B_2).
\end{equation}
Putting everything together, we obtain
\begin{align}
&\qquad \P_{x_0}(Y_1^{A^c \rangle} \in \bar B_1, Y_2^{A^c \rangle} \in B_2) \notag \\
&=\sum_{k=1}^\infty \int_{\bar B_1}  \P_{x_0}(T_1^{\rangle A}>k-1, Y_{k-1} \in d x_1) \P_{x_1}(Y_1 \in A) P_{A^c}^{ex}(x_1, B_2) \notag\\
&=\lim_{m \to \infty }  \int_{\bar B_1}  \left [ \sum_{k=1}^m \P_{x_0}(T_1^{\rangle A}>k-1, Y_{k-1} \in d x_1) \right ] \P_{x_1}(Y_1 \in A) P_{A^c}^{ex}(x_1, B_2) \notag\\
&=\lim_{m \to \infty }  \int_{\bar B_1}  \P_{x_0}(Y_1^{A^c \rangle} \in dx_1, T_1^{\rangle A} \le m)  P_{A^c}^{ex}(x_1, B_2) \notag\\
&=\int_{\bar B_1} \P_{x_0}(Y_1^{A^c \rangle} \in dx_1) \int_{B_2} P_{A^c}^{ex}(x_1, d x_2), \label{eq: Markov exit} 
\end{align}
where in the third equality we used that the measure $\P_{x_0}(Y_1^{A^c \rangle} \in dx_1, T_1^{\rangle A} \le m)$ has density $\P_{x_1}(Y_1 \in A)$ with respect to the finite measure $[\ldots]$ in the third line. It remains to notice that we can replace $\bar B_1$ by $B_1$ using that $B_1 \setminus \bar B_1 = \{\dagger \}$ and
\[
\P_{x_0}(Y_1^{A^c \rangle} = \dagger, Y_2^{A^c \rangle} \in B_2) = \P_{x_0}(T_1^{\rangle A} = \infty) \I_{B_2}(\dagger)
=\int_{\{\dagger\}} \P_{x_0}(Y_1^{A^c \rangle} \in dx_1) \int_{B_2} P_{A^c}^{ex}(x_1, d x_2).
\]
This proves the basis of induction.

To prove the inductive step, we proceed exactly as above and arrive at
\begin{align*}
&\mathrel{\phantom{=}}   \P_{x_0}(Y_1^{A^c \rangle} \in \bar B_1, \ldots, Y_{n+1}^{A^c \rangle} \in B_{n+1}) \\
 &=\sum_{k=1}^\infty \int_{\bar B_1} \P_{x_0}(T_1^{\rangle A}>k-1, Y_{k-1} \in d x_1) \int_A \P_z(Y_1^{A^c \rangle} \in B_2, \ldots, Y_n^{A^c \rangle} \in B_{n+1}) \,  \P_{x_1}(Y_1 \in d z).
\end{align*} 
Using the assumption of induction for the integrand under $\int_A$, we get
\[
\int_A \P_z(Y_1^{A^c \rangle} \in B_2, \ldots, Y_n^{A^c \rangle} \in B_{n+1}) \, \P_{x_1}(Y_1 \in d z) = \int_A \P_{x_1}(Y_1 \in d z) \int_{B_2} f(x_2) \P_z(Y_1^{A^c \rangle} \in d x_2), 
\]
where $f$ is a non-negative measurable function on $B_2$ given by
\[
f(x_2):= \int_{B_3} P_{A^c}^{ex}(x_2, d x_3) \ldots  \int_{B_{n+1}} P_{A^c}^{ex}(x_n, d x_{n+1}). 
\]

We claim that  for any $x_1 \in \bar B_1$ and any non-negative measurable function $g$ on $B_2$, 
\begin{equation} \label{eq: Pf}
\int_A \P_{x_1}(Y_1 \in d z) \int_{B_2} g(x_2) \P_z(Y_1^{A^c \rangle} \in d x_2)  = \P_{x_1}(Y_1 \in A) \int_{B_2} g(x_2) P_{A^c}^{ex}(x_1, d x_2).
\end{equation}
Indeed, for indicator functions $g$ this holds by definition~\eqref{eq: exit chain} of  $P_{A^c}^{ex}$; cf.~\eqref{eq: P1}. Hence, \eqref{eq: Pf} holds for simple functions (i.e.\ finite linear combinations of indicator functions) by additivity of the three integrals in~\eqref{eq: Pf}. Finally, since any non-negative measurable function $g$ can be represented as pointwise limit of a pointwise non-decreasing sequence of simple functions,  equality~\eqref{eq: Pf} follows from the monotone convergence theorem.

Putting everything together and applying \eqref{eq: Pf} with $g=f$ establishes the inductive step exactly as we obtained \eqref{eq: Markov exit} applying \eqref{eq: P1} in the case $n=2$ and then replacing $\bar B_1$ by $B_1$. 
\end{proof}

\section{Invariance by inducing for recurrent chains} \label{sec: stationarity MC}

In this section we study stationarity of general entrance and exit Markov chains using the methods of infinite ergodic theory. Our main results here concern recurrent chains. 

\subsection{Setup and notation} \label{sec: setup MCs}
Let $Y$ be a Markov chain on a measurable space $(\XX, \FF)$. Denote by $\PP_\nu^Y:=\P_\nu(Y \in \cdot)$ the measure on the space of sequences $(\XX^{\N_0}, \FF^{\otimes \N_0})$, which is the ``distribution'' of $Y$ started under a measure $\nu$, and denote by $\EE_\nu^Y$ the Lebesgue integral with respect to $\PP_\nu^Y$. 

For the rest of Section~\ref{sec: setup MCs} we assume that $\mu$ a  non-zero invariant measure of $Y$. 

Let $\theta$ be the (one-sided) {\it shift} operator on $\XX^{\N_0}$ defined by $\theta:(x_0, x_1, \ldots ) \mapsto (x_1, x_2, \ldots)$. This is a measure preserving transformation of the measure space $(\XX^{\N_0}, \FF^{\otimes \N_0}, \PP_\mu^Y)$. For a set $C \in \FF^{\otimes \N_0}$, consider the {\it first hitting time} $T_C$ of $C$ and the {\it induced shift} $\theta_C$ defined~by
\[
T_C(x):=\inf\{n \in \N: \theta^n x \in C\}, \,  x \in \XX^{\N_0}, \quad \text{and} \quad \theta_C(x):=\theta^{T_C(x)} x, \, x \in C \cap \{T_C < \infty\},
\]
where $\inf_\varnothing:=\infty$ by convention. These mappings are measurable.

The powerful idea of ergodic theory is that the induced shift $\theta_C$ is a measure preserving transformation of the induced space $(C, {(\FF^{\otimes \N_0})}_C, {(\PP_\mu^Y)}_C)$, under certain recurrence-type assumptions on $Y$ and $C$, e.g.\ as in Lemmas~\ref{lem: induced'} and ~\ref{lem: induced2} in the Appendix (where we also review the relevant notions of ergodic theory). Below we introduce the definitions needed to apply these general results of ergodic theory  in the context of Markov chains. We also refer the reader to  Kaimanovich~\cite[Section~1]{Kaimanovich} for a brief account of relevant results on invariant Markov shifts, and to Foguel~\cite{Foguel} for a detailed one.

Denote by $C_B:=\{x \in \XX^{\N_0}: (x_0, \ldots,  x_{k-1}) \in B\}$ the cylindrical set with a base $B \in \FF^{\otimes k}$, where $k \ge 1$, and put $\tau_B:= T_{C_B}$. An invariant measure $\mu$ of the Markov chain $Y$ is called {\it recurrent} if for every set $A \in \FF$ such that $\mu(A)<\infty$, we have $\P_x(\tau_A(Y)<\infty)=1$ for $\mu$-a.e.\ $x \in A$. It follows easily from the invariance of $\mu$ that this definition is equivalent to $\P_x(\{Y_n \in A \text{ i.o.}\})=1$ for $\mu$-a.e.\ $x \in A$; cf.~\eqref{eq: io}. Note that every finite invariant measure is recurrent by Poincar{\'e}'s recurrence theorem (see the Appendix).
Following Kaimanovich~\cite{Kaimanovich}, we say that an invariant measure $\mu$ of $Y$ is {\it transient} if for every $A \in \FF$ such that $\mu(A)<\infty$,  we have $\P_x(\{Y_n \in A \text{ i.o.}\})=0$ for $\mu$-a.e.\ $x \in A$. We stress that the latter condition can be violated when $Y$ is transient but $\mu(A)=\infty$. There is a usual transience--recurrence dichotomy, see Lemma~\ref{lem: cond rec}.\ref{cond: dychotomy} below.

Furthermore, we say that $\mu$ is {\it ergodic} if  the shift $\theta$ is ergodic (and $\theta$ is $\PP_\mu^Y$-preserving). We say that $\mu$ is {\it irreducible} if every invariant set of $Y$ is $\mu$-trivial, that is for any $A \in \FF$, the equality $\P_x(Y_1 \in A) = \I_A(x)$ for $\mu$-a.e.\ $x$ implies that either $\mu(A)=0$ or $\mu(A^c)=0$; this shall not be confused with the notion of {\it  $\mu$-irreducibility} of Markov chains considered e.g.\ in Meyn and Tweedie~\cite[Section~4.2]{MeynTweedie}. 

Let us give necessary and sufficient conditions for recurrence and ergodicity of $Y$. 

\begin{lemma} \label{lem: cond rec}
Let $Y$ be a Markov chain that takes values in a measurable space $(\XX,\FF)$ and has an invariant measure $\mu$ on $(\XX,\FF)$. 
\begin{enumerate}[a)]
\item \label{cond: conserv} If $\mu$ is $\sigma$-finite and recurrent for $Y$, then  the shift $\theta$ on $(\XX^{\N_0}, \FF^{\otimes \N_0}, \PP_\mu^Y)$ is conservative. Conversely, if $\theta$ is conservative, then $\mu$ is recurrent for $Y$.

\item \label{cond: recur} $\mu$ is $\sigma$-finite and recurrent for $Y$ if and only if there exists a sequence of sets $\{B_n\}_{n \ge 1} \subset \FF$ such that $\XX= \cup_{n \ge 1} B_n \Mod{\mu}$, and $\P_{\mu_{B_n}}(\tau_{B_n}(Y) = \infty) =0$ and $\mu(B_n)<\infty$ for every $n \ge 1$.

\item \label{cond: recur finite}  $\mu$ is recurrent for $Y$ if for some $k \ge 1$ there exists a set $B \in \FF^{\otimes k}$ such that $\P_\mu(\tau_B(Y) = \infty) = 0$ and $\P_\mu((Y_1, \ldots, Y_k) \in B) < \infty$.

\item \label{cond: ergodic} $\mu$ is ergodic and recurrent for $Y$ if and only if $\mu$ is irreducible and recurrent for $Y$.

\item \label{cond: dychotomy} If $\mu$ is $\sigma$-finite and irreducible for $Y$, then $\mu$ is either recurrent for $Y$ or transient for $Y$.


\end{enumerate}
\end{lemma}
 
\begin{proof}
\ref{cond: conserv}) For the direct implication, note that since $\mu$ is $\sigma$-finite, $\XX^{\N_0}$ can be exhausted by countably many cylindrical sets $C_{B_n}$ with bases $B_n \in \BB(\XX)$ of finite measure. Each  set has measure $\PP_\mu^Y(C_{B_n})=\mu(B_n) < \infty$ and is recurrent for $\theta$ by recurrence of $\mu$ for $Y$. Then $\theta$ is conservative by Lemma~\ref{lem: criteria}. For the reverse implication, every measurable cylindrical set $C_B$ is recurrent for $\theta$ by conservativity of $\theta$, hence $\mu$ is recurrent. 

\ref{cond: recur}) This follows as above. 

\ref{cond: recur finite}) The shift $\theta$ is conservative by  Lemma~\ref{lem: criteria} since $\PP_\mu^Y(C_B)= \P_\mu((Y_1, \ldots, Y_k) \in B) < \infty$ and $\PP_\mu^Y(T_{C_B} = \infty)=0 $, hence $\mu$ is recurrent for $Y$ by Part~\ref{cond: conserv}).

\ref{cond: ergodic}) The direct implication holds true since every $\theta$-invariant cylindrical set $C_B$ with one-dimensional  base $B \in \FF$ is $\PP_\mu^Y$-trivial. The reverse one is stated in~\cite[Proposition~1.7]{Kaimanovich}. Let us give a full proof, completing the details sketched in~\cite{Kaimanovich}. 

Let $A \in \XX^{\otimes \N_0}$ be a measurable set that is$\Mod \PP_\mu^Y$-invariant under $\theta$. Then $\I_A(Y)=\I_A(\theta^k Y)$ for every $k \ge 0$ $\P_\mu$-a.e., hence there exists a set $N \in \FF$ of full measure $\mu$ such that for every $x \in N$, $\I_A(Y)=\I_A(\theta^k Y)$ for every $k \ge 0$ $\P_x$-a.e. Therefore,  for every $x \in N$, we have $\E_x (\I_A(Y)|\FF_k)=\E_x (\I_A(\theta^k Y)| \FF_k) $ for all $k \ge 0$ $\P_x$-a.e., where $\FF_k:=\sigma(Y_0, \ldots, Y_k)$ and $\E_x( \cdot| \FF_k)$ denotes the conditional expectation with respect to $\P_x$. On the other hand, denoting $h(x):=\P_x (Y \in A)$ for $x \in \XX$, we have $h(Y_k) = \E_x (\I_A(\theta^k Y)| \FF_k) $ for all $k \ge 0$ $\P_x$-a.e.\ for every $x$ by the Markov property of $Y$. Therefore, for every $ x \in N$,
\begin{equation} \label{eq: martingale}
h(Y_k)=\E_x (\I_A(Y)|\FF_k) \text{ for every }k \ge 0, \qquad \P_x \text{-a.e.}
\end{equation}
Thus, $\{h(Y_k)\}_{k \ge 0}$ is a bounded $\P_x$-martingale, hence $\lim_{k \to \infty} h(Y_k)=\I_A(Y)$ $\P_x$-a.e., which yields $\lim_{k \to \infty} h(Y_k)=\I_A(Y)$ $\P_\mu$-a.e.

Note that $h(x) \in \{0,1\}$ for $\mu$-a.e.\ $x$. Otherwise, the set $B':=h^{-1}((\varepsilon, 1-\varepsilon))$  is $\mu$-non-zero for some $\varepsilon \in (0,1/2)$, and $\P_{\mu_{B'}}(\{Y_k \in B' \text{ i.o.}\})>0$ by recurrence of $\mu$ for $Y$, in contradiction with $\lim_{k \to \infty} h(Y_k)=\I_A(Y)$ $\P_\mu$-a.e. Then for $x \in N \setminus B'$, we have $h(Y_k)=h(x)$ for every $k \ge 0$ $\P_x$-a.e. This follows from $\E_x h(Y_k)=h(x) \in \{0,1\}$, which itself is a corollary to~\eqref{eq: martingale}. Hence  $\I_A(Y)=h(Y_0)$ $\P_\mu$-a.e.; note that this generalizes~\cite[Lemma~17.1.1]{MeynTweedie}. On the other hand, for $B:=h^{-1}(1)$, we have $h(x)= \I_B(x)=\P_x(Y \in A)$ for $\mu$-a.e.\ $x$. Hence either $h =0$ $\mu$-a.e.\ or $h =1$ $\mu$-a.e.\ since the measure $\mu$ is irreducible for $Y$, and therefore $\mu$ is ergodic. 

\ref{cond: dychotomy}) This is stated in~\cite[Theorem~1.2]{Kaimanovich} but neither a formal proof nor an exact reference is given. Let us show how this claim follows from the considerations by Foguel~\cite[Chapter~II]{Foguel}. We have $\XX = C \cup D$, where $C$ and $D$ are respectively the {\it conservative} and {\it dissipative} parts of the space, which are disjoint measurable sets defined in~\cite[Eq.~(2.2)]{Foguel}. It follows from the  irreducibility of $\mu$ for $Y$ that $\P_x(Y_1 \in C)=\I_C(x) \Mod{\mu}$; see~\cite[p.~17]{Foguel}. Then either $C=X \Mod{\mu}$, in which case $Y$ is recurrent by~\cite[Eq.~(2.4)]{Foguel}, or $D=X \Mod{\mu}$, in which case $Y$ is transient by the following argument (cf.~the one on p.~12 in~\cite{Foguel}). For any $B \in \FF$ such that $\mu(B)<\infty$, take $f= \I_B \in L^\infty(\mu)$ and $u = \I_{B_n} \in L^1(\mu)$ with $B_n:=\big \{x \in B: \sum_{k=0}^\infty \frac{d}{d \mu }\P_{\mu} (Y_k \in \cdot)(x) \le n \big \}$. Then, using the notation of  \cite[Chapter~I]{Foguel}, 
\[
\sum_{k=0}^\infty \P_{\mu_{B_n}} (Y_k \in B) = \Big \langle u ,  \sum_{k=0}^\infty P^k f \Big \rangle =  \Big \langle \sum_{k=0}^\infty u P^k,  f  \Big \rangle \le n \mu(B) < \infty.
\]
This implies that $\P_x(\{Y_n \in B \text{ i.o.}\})=0$ for $\mu$-a.e.\ $x \in B$ since $B_n$ increases to $B$ as $n \to \infty$ by~\cite[Eq.~(2.3)]{Foguel}. Thus, $\mu$ is transient for $Y$.
\end{proof}

\subsection{General recurrent Markov chains} \label{sec: Inducing MCs}

In the proof of the following result, we will use the method of inducing to {\it compute} invariant measures of the entrance and exit chains. 

\begin{theorem} \label{thm: inducing MC recurrent}
Let $Y$ be a Markov chain that takes values in a measurable space $(\XX, \mathcal F)$ and has a $\sigma$-finite recurrent invariant measure $\mu$ on $(\XX, \FF)$.  Let $A \in \mathcal F$ be a set such that $\P_\mu(Y_0 \in A^c,  Y_1 \in A)>0$. Then the measures
\begin{equation} \label{eq: measures def gen}
\mu_A^{entr}:=\int_{A^c} \P_x(Y_1 \in \cdot ) \mu(dx) \text{ on } (A,\FF_A) \quad \text{and} \quad \mu_{A^c}^{exit}(dx):= \P_x(Y_1 \in A) \mu(dx) \text{ on } (A^c,\FF_{A^c})
\end{equation}
are proper, recurrent, and invariant for the entrance chain $Y^{\rangle A}$ and the exit chain $Y^{A^c \rangle}$, respectively. They are ergodic if $\mu$ is irreducible, and in this case
\begin{equation} \label{eq: Kac MC}
\mu= \E_{\mu_A^{entr}} \Bigg [ \sum_{k=0}^{T_1^{\rangle A} - 1} \I(Y_k \in \cdot) \Bigg ].
\end{equation}
Moreover, equality~\eqref{eq: Kac MC} holds true when 
\begin{equation} \label{eq: Y crosses}
\P_{\mu_A}(\tau_{A^c}(Y)=\infty)=0 \quad \text{and} \quad \P_{\mu_{A^c}}(\tau_A(Y)=\infty)=0.
\end{equation}
\end{theorem} 

The measure $\mu_A^{entr}$ has a simpler form if the chain $Y$ has a dual relative to $\mu$, see Section~\ref{sec: duality}. Equation~\eqref{eq: Kac MC} is a particular case of {\it Kac's formula} of Lemma~\ref{lem: induced3}.
\begin{cor} \label{cor: alternation}
If~\eqref{eq: Kac MC} is satisfied, then $\P_{\mu_A^{entr}}(Y_{T_1^{\rangle A^c}} \in \cdot ) = \mu_{A^c}^{entr}$ and
\begin{equation} \label{eq: Kac MC split}
\mu= \E_{\mu_A^{entr}} \Bigg [ \sum_{k=0}^{T_1^{\rangle A^c} - 1} \I(Y_k \in \cdot) \Bigg ] + \E_{\mu_{A^c}^{entr}} \Bigg [ \sum_{k=0}^{T_1^{\rangle A} - 1} \I(Y_k \in \cdot) \Bigg ].
\end{equation}
\end{cor}
\begin{proof}
For any measurable $B \subset A^c$,
\begin{align} \label{eq: alternation}
\P_{\mu_A^{entr}}(Y_{T_1^{\rangle A^c}} \in B)&= \sum_{n=1}^\infty \P_{\mu_A^{entr}}(Y_n \in B, Y_{T_1^{\rangle A^c}} =n) \notag
\\ 
&= \sum_{n=1}^\infty \int_A \P_x(Y_1 \in B) \P_{\mu_A^{entr}}(Y_{n-1} \in dx, T_1^{\rangle A^c}>n-1) \notag\\
&= \int_A \P_x(Y_1 \in B) \, \E_{\mu_A^{entr}} \Bigg [ \sum_{k=0}^{T_1^{\rangle A^c} - 1} \I(Y_k \in dx ) \Bigg ] \notag\\
&= \int_A \P_x(Y_1 \in B) \mu(dx),
\end{align}
where the third equality follows from the monotone convergence theorem and the last equality follows from \eqref{eq: Kac MC}. Hence $\P_{\mu_A^{entr}}(Y_{T_1^{\rangle A^c}} \in \cdot ) = \mu_{A^c}^{entr}$ by the definition of $\mu_{A^c}^{entr}$. This implies~\eqref{eq: Kac MC split} by splitting the sum in~\eqref{eq: Kac MC} as $\sum_{k=0}^{T_1^{\rangle A} - 1} =\sum_{k=0}^{T_1^{\rangle A^c} - 1} + \sum_{k=T_1^{\rangle A^c}}^{T_1^{\rangle A} - 1}$ and using the strong Markov property of $Y$ for the second term.
\end{proof}

\begin{remark}
If we drop from Theorem~\ref{thm: inducing MC recurrent} the assumption that $\mu$ is recurrent but keep~\eqref{eq: Y crosses}, it remains true that $\mu_A^{entr}$ and $\mu_{A^c}^{exit}$ are proper excessive measures of the respective chains $Y^{\rangle A}$ and  $Y^{A^c \rangle}$, that is
\[
\int_A P_A^{entr}(x, \cdot)  \mu_A^{entr}(dx) \le \mu_A^{entr} \text{ on } (A, \FF_A) \quad \text{and} \quad \int_{A^c} P_{A^c}^{exit}(x, \cdot)  \mu_{A^c}^{exit}(dx) \le \mu_{A^c}^{exit} \text{ on } (A^c, \FF_{A^c}).
\]
The proof remains the same up to changing the first equalities in \eqref{eq: entr chain starts} and \eqref{eq: exit chain starts} to the inequalities `$\ge$' and using Lemma~\ref{lem: induced'} instead of Lemma~\ref{lem: induced2}.
\end{remark}

\begin{proof}[{\bf Proof of Theorem~\ref{thm: inducing MC recurrent}}]
Put $C:=C_{A^c \times A}$. The shift $\theta$ on $(\XX^{\N_0}, \FF^{\otimes \N_0}, \PP_\mu^Y)$ is measure-preserving, and  conservative by Lemma~\ref{lem: cond rec}.\ref{cond: conserv}. Then $\PP_\mu^Y(C \setminus \{\theta^k \in C \text{ i.o.}\})=0$ by~\eqref{eq: io}, and by definitions of the set $N_A$ and the measures $\mu_{A^c}^{exit}$ and $\mu_A^{entr}$,
\begin{align*}
0 &=  \P_\mu(Y_0 \in A^c \setminus N_A, Y_1 \in A ) + \P_\mu(Y_0 \in A^c, Y_1 \in A \setminus N_A ) = \mu_{A^c}^{exit}(A^c \setminus N_A) + \mu_A^{entr}(A \setminus N_A ).
\end{align*}
Thus, the measures $\mu_{A^c}^{exit}$ and $\mu_A^{entr}$ are proper for the chains $Y^{A^c \rangle }$ and $Y^{\rangle A}$, respectively.

Because $\theta$ is conservative and the cylindrical set $C$ with the two-dimensional base $A^c \times A$ satisfies $\PP_\mu^Y(C) = \P_\mu( Y_0 \in A^c, Y_1 \in A)>0$, the induced shift $\theta_{C}$  is a measure preserving transformation of the induced space $(C, {(\FF^{\otimes \N_0} )}_{C}, {(\PP_\mu^Y)}_{C})$ by Lemma~\ref{lem: induced2}. Then for any measurable set $B \subset A^c \times A$,
\begin{align*}
{(\PP_\mu^Y)}_{C}(C_B) &= \PP_\mu^Y \bigl(C \cap \{ \theta_{C}  \in C_B \} \cap \{\theta^k \in C \text{ i.o.}\} \bigr)\\
&= \P_{\mu} \bigl((Y_0, Y_1) \in A^c \times A,  \theta_{C}(Y) \in C_B, Y_k \in A \text{ i.o.}, Y_k \in A^c \text{ i.o.} \bigr),
\end{align*}
hence
\[
\P_\mu((Y_0, Y_1) \in B) = \P_{\mu} \bigl((Y_0, Y_1) \in A^c \times A, (Y_2^{A^c \rangle}, Y_2^{\rangle A}) \in B \bigr).
\]

Taking $B=A^c \times B_1$, where $B_1 \subset A$ is a measurable set, the above implies that
\begin{align} \label{eq: entr chain starts}
\mu_A^{entr}(B_1) &= \P_{\mu} \bigl((Y_0, Y_1) \in A^c \times A, Y_2^{\rangle A} \in B_1 \bigr) \notag\\
&= \int_{A^c} \mu(d x_0) \int_A \P_{x_0}(Y_2^{\rangle A} \in B_1 | Y_1 = x_1) \P_{x_0} (Y_1 \in d x_1) \notag\\
&= \int_{A^c } \mu(d x_0) \int_{A } \P_{x_1}(Y_1^{\rangle A} \in B_1) \P_{x_0} (Y_1 \in d x_1) \notag\\
&=\int_A P_A^{entr}(x_1, B_1)  \mu_A^{entr}(dx_1), 
\end{align}
where in the third equality we used the Markov property of $Y$ and in the last one we used formula \eqref{eq: entr kernel} for the transition kernel $P_A^{entr}$. Thus, $\mu_A^{entr}$ is invariant for $Y^{\rangle A}$. 

Similarly, let us take $B=B_0 \times A$, where $B_0 \subset A^c$ is an arbitrary measurable set. Then
\begin{align}  \label{eq: exit chain starts}
\mu_{A^c}^{exit}(B_0) &=\P_{\mu} \bigl((Y_0, Y_1) \in A^c \times A, Y_2^{A^c \rangle} \in B_0 \bigr) \notag\\
&= \int_{A^c} \mu(d x_0) \int_A \P_{x_1}(Y_1^{A^c \rangle} \in B_0) \P_{x_0} (Y_1 \in d x_1) \notag\\
&= \int_{A^c_{ex}} \P_{x_0} (Y_1 \in A) \mu(d x_0) \int_A \P_{x_1}( Y_1^{A^c \rangle} \in B_0 ) \P_{x_0} (Y_1 \in d x_1| Y_1 \in A) \notag\\
&= \int_{A^c} P_{A^c}^{exit}(x_0, B_0)  \mu_{A^c}^{exit}(dx_0), 
\end{align}
where in the last equality we used formula \eqref{eq: exit chain} for the transition kernel $P_{A^c}^{exit}$ of the entrance chain $Y^{A^c \rangle}$. Thus, $\mu_{A^c}^{exit}$ is invariant for $Y^{A^c \rangle}$.

For $i \in \{0,1\}$, define the mappings
\[
\qquad \qquad \qquad \psi_i(x):= (x_i, \theta_{C}(x)_i, (\theta_{C})^2(x)_i, \ldots), \qquad x \in C \cap \{\theta^k \in C \text{ i.o.}\},
\]
from their common domain to $(A^c)^\N$ and $A^\N$, respectively. These mapping are measurable. The entrance chain $Y^{\rangle A}$ starts from $Y_1^{\rangle A}$, which is $Y_1$ on the event  $\{Y_0 \in A^c, Y_1 \in A\}$. Moreover, $\psi_1(Y) = Y^{\rangle A}$ on  $\{Y_0 \in A^c, Y_1 \in A, Y_k \in A \text{ i.o.}, Y_k \in A^c \text{ i.o.}\}$. Therefore, 
since for any measurable set $B \subset A$ it is true that
\[
\mu_A^{entr}(B) = \P_\mu(Y_0 \in A^c, Y_1 \in B) = \PP_\mu^Y(x \in C: x_1 \in B ),
\]
we see that ${(\PP_\mu^Y)}_{C} \circ \psi_1^{-1}$ is
the law  on $(A^\N, \FF_A^{\otimes \N})$ of  $Y^{\rangle A}$ with $Y_1^{\rangle A}$ distributed according to $\mu_A^{entr}$. Denote this law by $\PP$. 

This representation of $\PP$ and the fact that $\PP_\mu^Y(C_{A^c \times B} \cap \{T_{C_{A^c \times B}}=\infty\}) =0$, which holds by conservativity of $\theta$ on $(\XX^{\N_0}, \FF^{\otimes \N_0}, \PP_\mu^Y)$, imply that the measure $\mu_A^{entr}$ is recurrent for   $Y^{\rangle A}$.

The shift $\theta_*$ on $A^\N$ is measure preserving on $(A^{\N},(\FF_A)^{\otimes \N} , \PP)$. If $E \in  (\FF_A)^{\otimes \N}$ is its invariant set, that is $\theta_*^{-1} E = E \Mod{\PP}$, then $\psi_1^{-1} (\theta_*^{-1} E) = \psi_1^{-1} E \Mod{{(\PP_{\mu}^Y)}_{C}}$. Note that $\psi_1(\theta_{C}(x))= \theta_*(\psi_1(x))$ for every $x \in C$, hence $\psi_1^{-1}(\theta_*^{-1}E) = \theta_{C}^{-1} (\psi_1^{-1}E) $, and therefore $\theta_{C}^{-1} (\psi_1^{-1} E) = \psi_1^{-1} E \Mod{{(\PP_\mu^Y)}_{C}}$, which means that $\psi_1^{-1} E$ is an invariant set for the induced shift $\theta_{C}$ on $C$. This set is ${(\PP_\mu^Y)}_{C}$-trivial, and thus $E$ is $\PP$-trivial, because $\theta_{C}$ is ergodic when the shift $\theta$ on $(\XX^\N, \FF^{\otimes \N}, \PP_\mu^Y)$ is ergodic and conservative by Lemma~\ref{lem: induced2}. This establishes ergodicity of the invariant measure $\mu_A^{entr}$ of the entrance chain $Y^{\rangle A}$ when $\mu$ is an ergodic invariant measure of $Y$. It remains to use that $\mu$ is ergodic when it is recurrent and irreducible; see Lemma~\ref{lem: cond rec}.\ref{cond: ergodic}.

Similarly, the law of the exit chain $Y^{A^c \rangle}$ with $Y_1^{A^c \rangle}$ following $\mu_A^{exit}$ is ${(\PP_\mu^Y)}_{C} \circ \psi_0^{-1}$. Ergodicity and recurrence of  $\mu_A^{exit}$ for $Y^{A^c \rangle}$ follow exactly as above.

To prove equality~\eqref{eq: Kac MC}, we will use Kac's formula of Lemma~\ref{lem: induced3} with $C$ substituted for $A$. To show that this result applies, we shall prove that $T_C$  is finite $\PP_\mu^Y$-a.e. This holds true by Lemma~\ref{lem: induced2} in the case when $\mu$ is irreducible (hence ergodic by Lemma~\ref{lem: cond rec}.\ref{cond: ergodic}), while under assumptions \eqref{eq: Y crosses} we argue as follows. Define $N_0:= \XX$ and for any $k \in \N$,
\[
N_k:= \bigcup_{B \in \{A, \, A^c\} } \Big\{ x \in N_{k-1} \cap B: \P_x \big(Y_{\tau_{B^c}(Y)} \in N_{k-1}, \tau_{B^c}(Y)<\infty\big)=1 \Big\}.
\]
Clearly, $N_0 \supset N_1 \supset N_2 \supset \ldots$, and it follows from the strong Markov property of $Y$ that when started from an $x \in N_k$, this chain crosses from $A$ to $A^c$ and from $A^c$ to $A$ at least $k$ times in total $\P_x$-a.s. Let us show by induction that
\begin{equation} \label{eq: N_k co-null}
\mu(N_k^c)=0, \qquad k \in \N.
\end{equation}
In fact, we have $\mu(N_1^c)=0$ by \eqref{eq: Y crosses}. For any $k \ge 2$, by  \eqref{eq: Y crosses},
\begin{align*}
\mu(N_{k-1} \setminus N_k) &= \sum_{B \in \{A, \, A^c\} } \int_{N_{k-1} \cap B}  \P_x(Y_{\tau_{B^c}(Y)} \not \in N_{k-1}) \mu(dx)\\ 
&\le \int_{N_{k-1}} \sum_{n=1}^\infty \P_x(Y_n  \not \in N_{k-1}) \mu(dx) \le \sum_{n=1}^\infty \mu(  N_{k-1}^c) 
\end{align*}
by $\mu$-invariance of $Y$. Hence $\mu(N_{k-1}^c)=0$ implies $\mu(N_k^c)=0$, and~\eqref{eq: N_k co-null} follows.

Then $T_C(Y)$ is finite $\P_\mu$-a.s.\ since it is finite $\P_x$-a.s.\ for every $x \in N_2$ and~\eqref{eq: N_k co-null} holds true. Therefore, Lemma~\ref{lem: induced3} applies. For any measurable $B \subset \XX$,
\begin{align} \label{eq: Kac proof}
\E_{\mu_A^{entr}} \Bigg [ \sum_{k=0}^{T_1^{\rangle A} - 1} \I(Y_k \in B) \Bigg ]
&= \int_{A^c} \mu(d x_0) \int_A \E_{x_1} \! \Bigg [\sum_{k=0}^{T_{C}(Y)} \I(Y_k \in B) \Bigg] \P_{x_0}(Y_1 \in d x_1) \notag \\
&= \int \limits_{A^c} \E_{x_0} \!  \Bigg [ \sum_{k=1}^{T_C(Y)} \I(Y_k \in B, Y_1 \in A ) \Bigg ]   \mu(d x_0),
\end{align}
where in the first equality we used the definition of $\mu_A^{entr}$ and the fact that $T_1^{\rangle A} = T_{C}(Y)  + 1$ on $\{Y_0 \in A \}$, and in the second equality we used the Markov property of $Y$. Finally, we obtain~\eqref{eq: Kac MC} by
\begin{equation} \label{eq: Kac final}
\E_{\mu_A^{entr}} \Bigg [ \sum_{k=0}^{T_1^{\rangle A} - 1} \I(Y_k \in B) \Bigg ] = \int \limits_{C}  \!  \Bigg [ \sum_{k=1}^{T_{C}(x)} \I(\theta^k x \in C_B) \Bigg ]  \PP_\mu^Y(dx)  =\PP_\mu^Y(C_B) = \mu(B),
\end{equation}
where in the second equality we applied Lemma~\ref{lem: induced3} after shifting  the summation indices by one using the invariance of ${(\PP_\mu^Y)}_{C}$ under the induced shift $\theta_{C}$.  
\end{proof}

Theorem~\ref{thm: inducing MC recurrent} above describes an invariant measure of the entrance chain obtained from a known invariant measure of the initial chain. The following result reverses this logic.

\begin{theorem} \label{thm: inducing MC reversed}
Let $Y$ be a Markov chain that takes values in a measurable space $(\XX, \mathcal F)$, and $A \in \FF$ be any non-empty set. Assume that $\nu$ is a measure on $(A, \mathcal F_A)$ such that $\P_{\nu}(T_1^{\rangle A}= \infty)=0$ and  $\nu$ is invariant for $Y^{\rangle A}$. Then $\nu$ is proper for $Y^{\rangle A}$ and the measure $\mu:= \E_{\nu} \big [ \sum_{k=0}^{T_1^{\rangle A} - 1} \I(Y_k \in \cdot) \big ] $ on $(\XX, \mathcal F)$ is invariant for $Y$. Moreover, we have $\nu = \mu_A^{entr}$ and $\P_\mu(T_1^{\rangle A}= \infty)=0$. Lastly, assume that  $\nu$ is $\sigma$-finite and recurrent for $Y^{\rangle A}$. Then $\mu$ is recurrent for $Y$; and if, in addition, $\nu$ is ergodic for $Y^{\rangle A}$, then $\mu$ is ergodic for $Y$. 
\end{theorem} 

\begin{proof}

The claim that the measure $\nu$ is proper for  $Y^{\rangle A}$ follows from its invariance and the assumption $\P_{\nu}(T_1^{\rangle A}= \infty)=0$. We omit this standard argument, which is similar to the one we used to prove~\eqref{eq: N_k co-null}.
The equality $\nu = \mu_A^{entr}$ follows exactly as in~\eqref{eq: alternation}, where we shall substitute $\mu_A^{entr}$ by $\nu$, interchange $A$ with $A^c$, and note that the last equality now defines $\mu$.

Define the measure $\mathrm P_0:=\P_{\nu}(\theta_C(Y) \in \cdot)$ on $(C, {(\FF^{\otimes \N_0})}_C)$, where $C=C_{A^c \times A}$, as before. This measure is invariant under the induced shift $\theta_C$ because for every measurable $E\subset C$,
\begin{equation} \label{eq: P_0 invariant}
\mathrm P_0(\theta_C \in E) = \P_{\nu}(\theta_C^2(Y) \in E) = 
\P_{\nu}(\theta_C(\theta (\theta_C(Y))) \in E)=
\P_{\nu}(\theta_C(Y) \in E) =\mathrm P_0,
\end{equation}
where in the third equality we combined the invariance of $\nu$ for $Y^{\rangle A}$ with the strong Markov property of $Y$, which implies that $\P_\nu(\theta (\theta_C(Y)) \in \cdot) = \P_{\nu'}(Y \in \cdot)$, where $\nu'=\P_\nu(Y_{T_1^{\rangle A}} \in \cdot)$. 

Furthermore, define the measure $\mathrm P:=\int_C \big [\sum_{k=1}^{T_C} \I(\theta^k(x) \in \cdot) \big] \mathrm P_0(dx)$ on $(\XX^{\N_0}, \FF^{\otimes \N_0})$. By Lemma~\ref{lem: Kac 2}, it is invariant for the shift $\theta$ on $(\XX^{\N_0}, \FF^{\otimes \N_0})$. Let us compute~$\mathrm P$. Similarly to \eqref{eq: P_0 invariant}, for every measurable $E \subset \XX^{\N_0}$ and $k \ge 1$,
\begin{align*}
\mathrm P_0(\theta^k \in E, T_C \ge k ) &= 
\P_{\nu}(\theta^k (\theta_C(Y)) \in E, T_C(\theta_C(Y)) \ge k) \\
&= \P_{\nu}(\theta^{k-1}(\theta(\theta_C(Y))) \in E, T_C(\theta(\theta_C(Y))) \ge k-1)\\
&= \P_{\nu}(\theta^{k-1}(Y) \in E, T_C(Y) \ge k-1).
\end{align*}
Hence,
\[
\mathrm P_0(\theta^k \in E, T_C \ge k ) = \int_\XX \P_x(Y \in E ) \P_{\nu}(Y_{k-1} \in dx, T_C(Y) \ge k-1),
\]
and by summing over $k \ge 1$ and using that  $T_1^{\rangle A}= T_C(Y)  + 1$ on $\{Y_0 \in A \}$, we arrive at
\[
\mathrm P(E) = \sum_{k=0}^\infty \int_\XX \P_x(Y \in E) \P_{\nu}(Y_{k-1} \in dx, T_1^{\rangle A} > k) = \int_\XX \P_x(Y \in E) \mu(dx) = \P_\mu(Y \in E),
\]
similarly to~\eqref{eq: alternation}. In particular, this implies that $\mu$ is an invariant measure of $Y$ since $\PP$ is invariant for $\theta$. By Lemma~\ref{lem: Kac 2}, we have $\P_\mu(T_1^{\rangle A}= \infty)=0$. 

Assume that the measure $\nu$ is $\sigma$-finite and recurrent for the entrance chain $Y^{\rangle A}$. Then $A$ can be covered $\mod{\nu}$ by a sequence of its measurable subsets ${\{A_n\}}_{n \ge 1}$ of finite measure $\nu$. Therefore, the measurable sets ${\{C_{A^c \times A_n}\}}_{n \ge 1}$ cover $C$ $\mod {\mathrm P_0}$, satisfy
\[
\mathrm P_0(C_{A^c \times A_n})=\P_{\nu}(Y_{\tau_C+1} \in A_n)=\P_{\nu}(Y_1^{\rangle A} \in A_n) = \nu(A_n)<\infty,
\]
and are recurrent for $\theta_C$ by recurrence of $\nu$ for $Y^{\rangle A}$. Hence  $\theta_C$ is a conservative measure preserving transformation of $(C, {(\FF^{\otimes \N_0})}_C, \mathrm P_0)$ by Lemma~\ref{lem: criteria}. Then $\theta$ is a conservative measure preserving transformation of $(\XX^{\N_0} , \FF^{\otimes \N_0}, \mathrm P)$ by Lemma~\ref{lem: Kac 2}, and $\mu$ is recurrent for $Y$ by Lemma~\ref{lem: cond rec}.\ref{cond: conserv}. 

Lastly, assume in addition that $\nu$ is  ergodic for $Y^{\rangle A}$. Then $\nu$ is irreducible for $Y^{\rangle A}$ by Lemma~\ref{lem: cond rec}.\ref{cond: ergodic}. By the same result, if $\mu$ is not ergodic for $Y$, then it is not irreducible, and there is a $\mu$-non-trivial set $B \in \FF$ such that $\I_B(x)=\P_x(Y_1 \in B)$ for $\mu$-a.e.\ $x$. Since $Y^{\rangle A}$ is a subchain of $Y$, we have $\I_B(x) \le \P_x(Y_1^{\rangle A} \in B)$ for $\mu$-a.e.\ $x$. By the same reasoning, $\I_{B^c}(x) \le \P_x(Y_1^{\rangle A} \in B^c)$ for $\mu$-a.e.\ $x$, hence $\I_B(x) = \P_x(Y_1^{\rangle A} \in B)$ for $\mu$-a.e.\ $x \in \XX$.
Then $\I_{A \cap B}(x)=\P_x(Y_1^{\rangle A} \in A \cap B)$ for $\mu_A$-a.e.\ $x \in A$.  Furthermore, the sets $A \cap B$ and $A^c \cap B$ are $\mu$-non-zero, because if e.g.\ $\mu(A^c \cap B)=0$, then it follows from the invariance of $\mu$ for $Y$ that $T_1^{\rangle A^c}= \infty$ $\P_\mu$-a.e., contradicting to $T_1^{\rangle A^c} < T_1^{\rangle A} < \infty$ $\P_{\mu_A}$-a.e.
Finally, it follows from $\nu = \int_{A^c} \P_x(Y_1 \in \cdot ) \mu(dx)$ that $\nu(A \cap B)>0$. By symmetry, $\nu(A \cap B^c)>0$, and thus $A \cap B$ is $\nu$-non-trivial, which is a contradiction because $\nu$ is irreducible for $Y^{\rangle A}$.
\end{proof}

\subsection{Weak Feller recurrent Markov chains} \label{sec: MC e and u}
In this section we give a topological counterpart to Theorem~\ref{thm: inducing MC recurrent}, assuming throughout that $\XX$ is a  metric space. 

We first give topological versions of the ergodic-theoretic definitions from Section~\ref{sec: setup MCs}. We say that a Markov chain $Y$ on $\XX$ is {\it topologically irreducible} if $\P_x(\tau_G(Y) < \infty )>0$ for every $x \in \XX$ and every non-empty open set $G \subset \XX$. 
We say that $Y$ is {\it topologically recurrent} if $\P_x(\tau_G(Y)<\infty)=1$ for every non-empty open set $G \subset \XX$ and {\it every} $x \in G$. We warn that the Markov chains literature often defines the topological recurrence by taking every $x \in \XX$ in the latter definition; see Lemma~\ref{lem: conventional topological recurrence} below regarding relationship between the definitions. 

The chain $Y$ is called  {\it weak Feller} if its transition probability $\P_x(Y_1 \in \cdot)$ is weakly continuous in $x$. Equivalently, the mapping $x \mapsto \E_x f(Y_1)$ is  continuous on $\XX$ for any continuous bounded function $f : \XX \to \R$.

A Borel measure on $\XX$ is called {\it locally finite} if every point of $\XX$ admits an open neighbourhood of finite measure. Such measure is finite on compact sets.
Moreover, it is $\sigma$-finite if $\XX$ is separable. Indeed, in this case $\XX$ has the Lindel\"of property  (see~Engelking~\cite[Corollary~4.1.16]{Engelking}), that is
\begin{equation} \label{eq: Lindelof}
\text{every open cover of a separable metric space contains a countable subcover}.
\end{equation}
Therefore, such $\XX$ can be represented as a countable union of open balls of finite measure. 

Our main result on weak Feller chains is as follows.



\begin{theorem} \label{thm: inducing bijection}
Let $Y$ be a topologically irreducible topologically recurrent weak Feller Markov chain that takes values in a separable metric space $\XX$.  Let $A \subset \XX$ be a Borel set such that $\P_x(Y_1 \in \Int(A))>0$ for some $x \in \Int(A^c)$. Then the mapping $\mu \mapsto \mu_A^{entr}$ (resp., $\mu \mapsto \mu_{A^c}^{exit}$), defined in \eqref{eq: measures def gen}, is a bijection between the sets of locally finite invariant Borel measures of the chain $Y$ on $\XX$ and the entrance chain $Y^{\rangle A}$ on $A$ (resp., the exit chain $Y^{A^c \rangle}$ on $A^c$).
\end{theorem}

The role of the condition $\P_x(Y_1 \in \Int(A))>0$ for an $x \in \Int(A^c)$ is to exclude the case where the chain $Y$ can enter $\Int A$ from its complement only through $\partial A$.

The main use of Theorem~\ref{thm: inducing bijection} is when the initial chain $Y$ is known to have a unique (up to a constant factor) locally finite invariant measure. In particular, this is true for any irreducible recurrent chain if the state space $\XX$ is discrete, by Revuz~\cite[Theorem~3.1.9]{Revuz}. This is as well true for recurrent random walks on $\R^d$, which we study below in Section~\ref{Sec: Application to RWs}. It is remarkable that under the assumptions of Theorem~\ref{thm: inducing bijection}, the chain $Y$ may have two non-proportional invariant measures even if the space $\XX$ is compact; see Skorokhod~\cite[Example~1]{Skorokhod} and also a simpler example by Carlsson~\cite[Theorem~1]{Carlsson}, where the assumptions are satisfied by Lemma~\ref{lem: conventional topological recurrence} below. 

The question of whether a weak Feller chain has a (non-zero) locally finite invariant measure was studied by Lin~\cite[Theorem~5.1]{Lin} and Skorokhod~\cite[Theorem~3]{Skorokhod}; the approach of \cite{Skorokhod} was similar to the one used here. They showed that under assumptions of Theorem~\ref{thm: inducing bijection}, the answer is positive when $\XX$ is a locally compact Polish space. The case of non-locally compact spaces was studied by Szarek~\cite{Szarek}. For existence and uniqueness results on invariant measures under much stronger assumptions on $Y$, such as strong Feller or Harris properties or $\psi$-irreducibility, see Foguel~\cite[Chapters~IV and VI]{Foguel} and Meyn and Tweedie~\cite[Chapter~10]{MeynTweedie}.

\medskip

Before proceeding to the proof of Theorem~\ref{thm: inducing bijection}, we give two simple auxiliary results. 

\begin{lemma} \label{lem: Radon}
Let $Y$ be a topologically irreducible weak Feller Markov chain that takes values in a metric space $\XX$ and has a non-zero  invariant Borel measure $\mu$. Then $\mu$ is locally finite if and only if it is finite on some non-empty open set. Moreover, if $\XX$ is separable, then $\mu$ is strictly positive on every non-empty open set.
\end{lemma}

\begin{proof}
The necessary condition for the local finiteness is trivial. To prove the sufficient one, assume that $G$ is a non-empty open subset of $\XX$ satisfying $ \mu(G)<\infty$. By the topological irreducibility of $Y$, for any $x \in \XX$ there exists an $n=n(x) \ge 1$ such that $\P_x(Y_n \in G)>0$. It follows by a simple inductive argument that the $n$-step transition probability $\P_x(Y_n \in \cdot)$ is weakly continuous in~$x$. Indeed, for any continuous bounded function $f: \XX \to \R$, we have 
\[
\E_x f(Y_n) = \int_\XX \E_y f(Y_{n-1}) \P_x( Y_1 \in dy), \qquad x \in \XX
\]
by the Chapman--Kolmogorov equation. The integrand is a continuous bounded function by assumption of induction, and so is the integral since $Y$ is weak Feller. 

Then there is an open neighbourhood $U_x$ of $x$ such that $\P_y(Y_n \in G) \ge \frac12 \P_x(Y_n \in G)>0$ for every $y \in U_x$. By the invariance of $\mu$, this gives 
\begin{equation} \label{eq: finiteness}
\infty > \mu(G) = \int_\XX \P_y(Y_n \in G) \mu(dy) \ge \int_{U_x} \P_y(Y_n \in G) \mu(dy) \ge \frac12 \P_x(Y_n \in G) \mu(U_x),
\end{equation}
implying finiteness of $\mu(U_x)$, as required.

Assume now by contraposition that $\XX$ is separable but there is a non-empty open subset $G$ of $\XX$ satisfying $ \mu(G)=0$. Recall that $\supp \mu$, the {\it topological support} of $\mu$, is the complement of the union of all open sets of zero measure $\mu$. Since $\XX \setminus \supp \mu$ is a separable metric space, by~\eqref{eq: Lindelof} it is a countable union of such open sets. Hence $\mu(\XX \setminus \supp \mu) =0$, and we have $\mu(\XX)=\mu(\supp \mu)>0 $. Therefore, there is at least one point $x \in \supp \mu$. Since $Y$ is topologically irreducible, there exists an $n \ge 1$ such that $\P_x(Y_n \in G)>0$. Picking an open neighbourhood $U_x$ of $x$ as above, from~\eqref{eq: finiteness} with $0$ substituted for $\infty$, we get $\mu(U_x)=0$. This contradicts to $x \in \supp \mu$.
\end{proof}

\begin{lemma} \label{lem: conventional topological recurrence}
Let $Y$ be a topologically irreducible topologically recurrent weak Feller Markov chain that takes values in a metric space $\XX$. Then
\begin{equation} \label{eq: conventional topological recurrence}
\P_x(\tau_G(Y) <\infty)=1 \text{ for every } x \in \XX \text{ and non-empty open }G \subset \XX.
\end{equation}
\end{lemma}
\begin{proof}
Fix an $x \in \XX$ and a non-empty open set $G$. As in the proof of Lemma~\ref{lem: Radon}, by the topological irreducibility and the weak Feller property of $Y$ we can find an open neighbourhood $U$ of $x$ such that $a:=\inf_{y \in U} \P_y(Y_n \in G)>0$ for some $n \ge 1$. Then $\sup_{y \in U} \P_y(\tau_G(Y) =\infty)\le 1-a$, and by the topological recurrence and the strong Markov property of $Y$, we have
\begin{align*}
\P_x(\tau_G(Y)=\infty) &= \P_x(\tau_U(Y)<\infty, \tau_G(Y)=\infty) \\
&=\int_U \P_y(\tau_G(Y)=\infty) \P_x(Y_{\tau_U(Y)} \in dy, \tau_U(Y)<\infty, \tau_G(Y)=\infty) \\
& \le (1-a) \P_x(\tau_G(Y)=\infty),
\end{align*}
which implies that $\P_x(\tau_G(Y)=\infty)=0$.
\end{proof}

\begin{proof}[{\bf Proof of Theorem~\ref{thm: inducing bijection}}] 
Denote by $L$, $L^{en}$, $L^{ex}$ the sets of non-zero locally finite Borel invariant measures of $Y$, $Y^{\rangle A}$, $Y^{A^c \rangle}$ on $\XX$, $A$, $A^c$, respectively. Similarly, denote by $L_0$, $L^{en}_0$, $L^{ex}_0$ the sets of locally finite Borel invariant measures of the respective chains; these sets are always non-empty.

Let $\mu \in L$. Since $\P_x(Y_1 \in \Int(A))>0$ for an $x \in \Int(A^c)$, by the weak Feller property of $Y$, there is an open neighbourhood $U \subset \Int(A^c)$ of $x$ such that $\P_y(Y_1 \in \Int(A))>0$ for every $y \in U$. We have $\mu(U)>0$ by Lemma~\ref{lem: Radon}, therefore $\P_\mu(Y_0 \in A^c, Y_1 \in A) \ge \P_\mu(Y_0 \in U, Y_1 \in \Int(A))>0$. Hence the measures $\mu_A^{entr}$ and $\mu_{A^c}^{exit}$ are non-zero, and they are locally finite by  $\mu_A^{entr} \le \mu_A$ and $\mu_{A^c}^{exit} \le \mu_{A^c}$. Furthermore, $\mu$ is $\sigma$-finite as a locally finite measure on a separable metric space; see~\eqref{eq: Lindelof}. By choosing an open set $G$ in~\eqref{eq: conventional topological recurrence} of finite measure, we conclude that $\mu$ is recurrent for $Y$ by Lemma~\ref{lem: cond rec}.\ref{cond: recur finite}. Therefore, Theorem~\ref{thm: inducing MC recurrent} applies, and the measures $\mu_A^{entr}$ and  $\mu_{A^c}^{exit}$ are invariant for the respective chains $Y^{\rangle A}$  and $Y^{A^c \rangle}$. Thus, the mappings $\mu \mapsto \mu_A^{entr}$ and $\mu \mapsto \mu_A^{exit}$ are from $L_0$ into $L^{en}_0$ and $L^{ex}_0$, respectively. 

We first consider the mapping $\mu \mapsto \mu_A^{entr}$ from $L_0$ into $L^{en}_0$. It is injective by \eqref{eq: Kac MC}, which holds true because~\eqref{eq: Y crosses} is satisfied thanks to
\[
\P_x(\tau_A(Y) \le \tau_{\Int(A)}(Y)< \infty) =1 \quad \text{and} \quad \P_x( \tau_{A^c}(Y) \le \tau_{\Int(A^c)}(Y)< \infty)=1, \qquad x \in \XX.
\]
To prove the surjectivity, let $\nu \in L^{en}$. From the equalities above, we get  $\P_x(T_1^{\rangle A}<\infty)=1$ for every $x \in \XX$. Therefore, by Theorem~\ref{thm: inducing MC reversed}, the measure $\mu:= \E_{\nu} \big [ \sum_{k=0}^{T_1^{\rangle A} - 1} \I(Y_k \in \cdot) \big ] $ is invariant for $Y$ and satisfies $\nu=\mu_A^{entr}$. It is non-zero by $\mu_A \ge \nu$. It remains to show that $\mu$ is locally finite, and thus $\mu \in L^{en}$.  

By the assumption, we have $\P_x(Y_1 \in \Int(A))>0$ for some $x \in \Int(A^c)$. Let us show that there exists an open set $G \subset \Int(A)$ such that $\nu(G)<\infty$ and $\P_x(Y_1 \in G)>0$. Indeed, since $\nu$ is a locally finite measure on $(A, \BB(A))$, the separable metric space $A$ is a countable union of sets $U$ of finite measure $\nu$ that are open in the topology of $A$; see~\eqref{eq: Lindelof}. Then $\Int(A)$ is a countable union of the sets $U \cap \Int(A)$ of finite measure $\nu$ that are open in the topology of $\XX$.  At least one of these sets must satisfy $\P_x(Y_1 \in U \cap \Int(A))>0$, otherwise $\P_x(Y_1 \in \Int(A))=0$ by the sub-additivity. 

By the weak Feller property of $Y$, we can find an open set $U_x$ such that $x \in U_x \subset \Int(A^c)$ and $\P_y(Y_1 \in G) \ge \frac12 \P_x(Y_1 \in G)$ for every $y \in U_x$. This yields $\mu(U_x) < \infty$ by~\eqref{eq: finiteness} with the first equality replaced by $\nu(G) = \int_{A^c} \P_x(Y_1 \in G) \mu(dx)$. Hence the measure $\mu$ on $\XX$ is locally finite by Lemma~\ref{lem: Radon}, as claimed.

Now consider the mapping $\mu \mapsto \mu_{A^c}^{exit}$ from $L_0$ to $L_0^{ex}$. To prove its surjectivity, let $\nu_0 \in L^{ex}$. One can show that the Borel measure $\nu:=\int_{A^c_{ex}} \P_y(Y_1 \in \cdot | Y_1 \in A) \nu_0(dy)$ on $A$ is invariant for the entrance chain $Y^{\rangle A}$ from $A^c$ to $A$. By Theorem~\ref{thm: inducing MC reversed}, the measure $\mu$, defined above, is invariant for the chain $Y$, and we have $\nu = \mu_A^{entr}$. Moreover,
\begin{align} \label{eq: nu_exit}
\P_y(Y_1 \in A) \mu(dy) &= \sum_{k=0}^\infty \P_y(Y_1 \in A ) \P_\nu(Y_k \in dy, T_1^{\rangle A} >k), \qquad y \in A^c, \notag \\
&= \P_\nu(Y_{ T_1^{\rangle A} - 1} \in dy) = \P_\nu(Y_1^{A^c \rangle})=\nu_0(dy). 
\end{align}
This means that $\nu_0=\mu_{A^c}^{exit}$. We need to show that $\mu$ is locally finite. 

By the assumption, there exists an $x \in A^c$ is such that $\P_x(Y_1 \in \Int(A))>0$. Since $\nu_0$ is  locally finite on $A^c$, we can choose a set $U'$ open in the topology of $A^c$ such that $x \in U'$ and $\nu_0(U')<\infty$. Since $U' \cap \Int(A)$ is open in the topology of $\XX$, by the weak Feller property of $Y$, there is a set $U \subset U' \cap \Int(A)$ open in $\XX$ such that $x \in U$ and $\P_y(Y_1 \in \Int(A)) \ge \frac12 \P_x(Y_1 \in \Int(A)) $ for every $y \in U$. By \eqref{eq: nu_exit}, this gives 
\[
\mu(U) = \int_U \frac{\nu_0(dy)}{\P_y(Y_1 \in A) } \le \int_U \frac{\nu_0(dy)}{\P_y(Y_1 \in \Int(A)) } \le \frac{2 \nu_0(U)}{\P_x(Y_1 \in \Int(A)) } < \infty,
\]
hence $\mu$ is locally finite by Lemma~\ref{lem: Radon}. Thus, the mapping $\mu \mapsto \mu_{A^c}^{exit}$ is surjective. To prove its injectivity, assume that $\nu_0= \eta_{A^c}^{exit}$ for some $\eta \in L^{ex}$. By \eqref{eq: nu_exit}, for any Borel $B \subset A$, 
\begin{equation*}
\eta_{A}^{entr}(B)=\int_{A^c} \P_x(Y_1 \in B) \eta(d x)= \int_{A^c_{ex}} \P_x(Y_1 \in B |Y_1 \in A) \nu_0(d x)  = \nu(B) = \mu_A^{entr}(B),
\end{equation*}
hence $\eta=\mu$ by injectivity of the mapping $\mu \mapsto \mu_A^{entr}$. 
\end{proof}



\section{Invariance by duality} \label{sec: duality}

In this section we study invariant measures of the entrance and the exit chains derived from a Markov chain that is no longer assumed to be recurrent. Instead, we need to make additional assumptions in terms of the dual chain. 
We present our proofs using the probabilistic notation but essentially we employ inducing for {\it invertible} measure preserving two-sided Markov shifts.

Recall that probability transition kernels $P$ and $\hat P$ on $(\XX, \FF)$ are {\it dual} relative to a $\sigma$-finite measure $\mu$ on $(\XX, \FF)$ if
\begin{equation} \label{eq: detailed bal}
\mu(dx) P(x, dy) = \mu(dy) \hat P(y, dx), \qquad x, y \in \XX.
\end{equation}
This equality of measures on $(\XX \times \XX, \FF \otimes \FF)$ is called the {\it detailed balance condition}. It implies, by integration in $x$ or in $y$, that the measure $\mu$ is invariant for both  $P$ and $\hat P$. 

If $\XX$ is a Polish space, then any transition kernel on $\XX$ with a $\sigma$-finite invariant measure $\mu$ always has a dual kernel $\hat P$ relative to $\mu$. Indeed, if $\mu$ is a probability measure, then this claim is nothing but the disintegration theorem combined with existence of regular conditional distributions for probability measures on Polish spaces; see Kallenberg~\cite[Theorems~6.3,~6.4,~A1.2]{Kallenberg} or Aaronson~\cite[Theorem~1.0.8]{Aaronson}. This easily extends to $\sigma$-finite measures by $\sigma$-additivity. Note that if $\hat P'$ is another transition kernel dual to $P$ relative to $\mu$, then $\hat P'(x, \cdot)=\hat P(x, \cdot)$ for $\mu$-a.e.\ $x$ by Lemma~4.7 in Chapter~2 in Revuz~\cite{Revuz}.

Two Markov chains $Y$ and  $\hat Y$ on $\XX$ are {\it dual} relative to $\mu$ if so are their transition kernels. In other words, 
we have the {\it time-reversal  equality} 
\begin{equation*} 
\P_\mu((Y_0, Y_1) \in B) = \PP_\mu^{\hat Y}(y: (y_1, y_0) \in B), \qquad B \in \FF \otimes \FF,
\end{equation*}
where $y=(y_0, y_1, \ldots) \in \XX^{\N_0}$ and the r.h.s.\ refers to the realization of $\hat Y$ as the identity mapping on the canonical space $(\XX^{\N_0}, \FF^{\otimes \N_0}, \PP_\mu^{\hat Y})$. More generally, for any $k \ge 1$ we have
\begin{equation} \label{eq: time-reversal}
\P_\mu((Y_0, \ldots, Y_k) \in B) = \PP_\mu^{\hat Y}((y_k, \ldots,   y_0) \in B), \qquad B \in \FF^{\otimes (k+1)}.
\end{equation}

We now state the main result of the section.

\begin{theorem} \label{thm: invariant dual}
Let $Y$ be a Markov chain that takes values in a Polish space $\XX$ and has a $\sigma$-finite invariant Borel measure $\mu$. Then there exists a Markov chain $\hat Y$ with values in $\XX$ that is dual to $Y$ relative to $\mu$. Furthermore, let $A \in \BB(\XX)$ be a set such that
\begin{equation} \label{eq: dual assumptions}
\PP_{\mu_A}^{Y}(\tau_{A^c}=\infty)= \PP_{\mu_{A^c}}^{Y}(\tau_A=\infty)=\PP_{\mu_A}^{\hat Y}(\tau_{A^c}=\infty)= \PP_{\mu_{A^c}}^{\hat Y}(\tau_A=\infty)=0.
\end{equation}
Then  the measures $\mu_A^{entr}$ and $\mu_{A^c}^{exit}$, defined in \eqref{eq: measures def gen}, are proper and invariant for the entrance chain $Y^{\rangle A}$ and the exit chain $Y^{A^c \rangle}$, respectively; and we have 
\begin{equation} \label{eq: measures symm}
\mu_A^{entr}(dx)=\hat P(x, A^c) \mu(dx), \qquad x \in A,
\end{equation}
where $\hat P$ denotes the transition kernel of $\hat Y$. Moreover, Kac's formula~\eqref{eq: Kac MC} holds true.
\end{theorem}

We will prove Theorem~\ref{thm: invariant dual} as an easy corollary to the following duality result. Recall that $N_A(Y)$, defined in \eqref{eq: N}, is the set of all points starting from where $Y$ visits both sets $A$ and $A^c$ infinitely often.

\begin{proposition} \label{prop: duality}
Let $Y$ and $\hat Y$ be Markov chains with values in a measurable space $(\XX, \FF)$ that are dual relative to a $\sigma$-finite measure $\mu$. Let $A \in \FF$ be a set such that
\begin{equation} \label{eq: N and N}
\mu  \big (N_A(Y) \Delta N_A(\hat Y)  \big ) =0.
\end{equation}
Then the exit chain $Y^{A^c \rangle}$ and the entrance chain $\hat Y^{\rangle A^c}$ are dual relative to the measure 
\begin{equation} \label{eq: dual 1}
\tilde \mu_{A^c}^{exit}(B):=\mu_{A^c}^{exit}(B \cap N_A(Y)), \qquad B \in \FF^\dagger_{A^c}.
\end{equation}
Likewise, the chains $Y^{\rangle A}$ and ${\hat Y}^{A \rangle}$ are dual relative to the measure
\begin{equation} \label{eq: dual 2}
\tilde \mu_{A}^{entr}(B):= \mu_{A}^{entr}(B \cap N_A(Y)), \qquad B \in \FF^\dagger_A.
\end{equation}
Moreover, it is true that
\begin{equation} \label{eq: dual alternation}
\P_{\tilde \mu_{A}^{entr}}(Y^{\rangle A^c}_1 \in \cdot) = \tilde \mu_{A^c}^{entr} \qquad \text{and} \qquad \P_{\tilde \mu_{A^c}^{entr}}(Y^{\rangle A}_1 \in \cdot) = \tilde \mu_{A}^{entr} .
\end{equation}
\end{proposition} 

In the special case when $Y$ is a one-dimensional oscillating  random walk $S$ on $\XX=\ZZ$ and $A=[0,\infty) \cap \ZZ$, we can write~\eqref{eq: dual alternation} as $\P_{\pi_+}(O^\downarrow_1 \in \cdot ) = \pi_-$ and $\P_{\pi_-}(O_1 \in \cdot ) = \pi_+$, where $O_1^\downarrow:= \mathcal{O}_{2 - \I(S_0\ge 0)}$ is the first overshoot at down-crossing of zero, $\pi_+:= \pi|_A$ and $\pi_-:=\pi|_{A^c}$ with $\pi$ is defined in \eqref{eq: pi}. These equalities were proved in~\cite[Remark~2.2]{MijatovicVysotsky}. Moreover, if $\ZZ=\R$, we can complement the second duality in Proposition~\ref{prop: duality} by a surprising representation of the transition probabilities of the chains $S^{\rangle A}$ and $-S^{A^c \rangle}$ (i.e.\ $O$ and $-U$) as products of two transition probabilities that are {\it reversible} relative to $\lambda_A^{entr}$ (i.e.\ $ \pi_+$); see \cite[Section~2.4]{MijatovicVysotsky}.

\begin{proof}[{\bf Proof of Proposition~\ref{prop: duality}}]
To stress that the measures $\mu_{A}^{entr}$ and $\mu_{A^c}^{exit}$ are defined in~\eqref{eq: measures def gen} in terms of the chain $Y$, we us write $\mu_{A,Y}^{entr}$ and $\mu_{A^c,Y}^{exit}$. For any measurable set $B \subset A$, 
\begin{equation} \label{eq: mu^entr transf}
\mu_A^{entr}(B) = \int_{A^c} \P_x(Y_1 \in B) \mu(dx) =  \P_\mu(Y_0 \in A^c, Y_1  \in B) = \PP_\mu^{\hat Y}  (y_0 \in B, y_1 \in A^c), 
\end{equation}
hence $\mu_{A,Y}^{entr}(dx)=\hat P(x,  A^c) \mu(dx)$ for $x \in A$. Thus $\mu_{A,Y}^{entr} = \mu_{A,\hat Y}^{exit}$, that is the entrance measure of $Y$ into $A$ from $A^c$  is the exit measure of $\hat Y$ exiting from $A$ to $A^c$. Then  we also have $\tilde \mu_{A,Y}^{entr} = \tilde \mu_{A,\hat Y}^{exit}$ by~\eqref{eq: N and N}. Therefore,  duality \eqref{eq: dual 2} follows from duality \eqref{eq: dual 1} applied to $\hat Y$ and $A$ in place of $Y$ and $A^c$.

To prove duality \eqref{eq: dual 1}, we need to check the detailed balance condition
\begin{equation} \label{eq: det bal dagger}
\tilde \mu_{A^c}^{exit}(dx) P_{A^c}^{exit}(x, dy) = \tilde \mu_{A^c}^{exit}(dy) \hat P_{A^c}^{entr}(y, dx), \qquad x, y \in A^c_\dagger,
\end{equation}
where $P_{A^c}^{exit}$ and $\hat P_{A^c}^{entr}$ denote the transition kernels of the chains $Y^{A^c \rangle}$ and $\hat Y^{\rangle A^c}$, and recall that by convention, $P_{A^c}^{exit}(x, \{ \dagger \} )=1$  for $x \in A^c \setminus A^c_{ex}(Y)$.  We will use the simplified notation $N=N_A(Y)$ and $\hat N = N_A(\hat Y)$.

If $x =\dagger$, then the l.h.s.\ of~\eqref{eq: det bal dagger} is zero by the definition of $\tilde \mu_{A^c}^{exit}$ and the r.h.s.\  of~\eqref{eq: det bal dagger} is zero since $\hat P_{A^c}^{entr}(y, \{ \dagger \})=0$ for every $y \in A^c \cap \hat N$ (by Lemma~\ref{lem: Markov}) and $\tilde \mu_{A^c}^{exit}$ is supported on $A^c \cap \hat N$ by assumption~\eqref{eq: N and N}. Thus, equality \eqref{eq: det bal dagger} is satisfied when $x =\dagger$. Similarly,~\eqref{eq: det bal dagger} is true when $y =\dagger$, in which case the l.h.s.\ is zero since $\tilde \mu_{A^c}^{exit}$ is supported on $A_{ex}^c(Y) \cap N$. Thus, we need to establish \eqref{eq: det bal dagger} only for $x, y \in A^c$. By the definition of $\tilde \mu_{A^c}^{exit}$, this amounts to showing that for any measurable sets $B_1, B_2 \subset A^c$,
\begin{equation} \label{eq: main duality}
\int_{B_1 \cap N} P_{A^c}^{exit}(x, B_2)  \P_x( Y_1 \in A) \mu(dx) = \int_{B_2 \cap N} \hat P_{A^c}^{entr}(y, B_1)  \P_y( Y_1 \in A) \mu(dy). 
\end{equation}

By formula \eqref{eq: exit chain} for the transition kernel $P_{A^c}^{exit}$ and the absorbing property~\eqref{eq: NA} of $ N$, 
\begin{align*} 
\text{LHS } \eqref{eq: main duality} &=\int_{B_1 \cap N} \mu(dx) \int_A \P_z(Y_1^{A^c \rangle} \in B_2)  \P_x(Y_1 \in dz)  \\
&= \sum_{k,m=1}^\infty \P_\mu \big((Y_n)_{n=0}^{k+m+1} \in (B_1  \cap N )  \times A^k \times (A^c)^{m-1} \times (B_2 \cap  N )  \times A \big).
\end{align*}
In the last line, we can replace $N$ by $\hat N$ on both occasions using assumption \eqref{eq: N and N} and invariance of $\mu$ for $Y$. Next we apply duality relation~\eqref{eq: time-reversal} to obtain that
\begin{align*}
\text{LHS } \eqref{eq: main duality} 
&= \sum_{k,m=1}^\infty \PP^{\hat Y}_\mu \big((y_n)_{n=0}^{k+m+1} \in 
A \times (B_2 \cap \hat N)  \times (A^c)^{m-1} \times A^k \times (B_1 \cap \hat N)  \big), \\
&= \int_A \mu(dx) \int_{B_2 \cap \hat N}  \sum_{k,m=1}^\infty \PP^{\hat Y}_z  \big((y_n)_{n=0}^{k+m-1} \in 
(A^c)^{m-1} \times A^k \times (B_1 \cap \hat N)  \big) \, \PP^{\hat Y}_x (y_1 \in dz),
\end{align*}
and noting that the sum in the last line is $\hat P_{A^c}^{entr}(z, B_1)$ by \eqref{eq: entr kernel} and the absorbing property~\eqref{eq: NA} of $\hat N$, we arrive at
\[
\text{LHS } \eqref{eq: main duality}  = \E^{\hat Y}_\mu [\I_A(y_0) \I_{B_2 \cap \hat N}(y_1) \hat P_{A^c}^{entr}(y_1, B_1)].
\]
By duality of $Y$ and $\hat Y$ with respect to $\mu$, this gives the required equality
\[
\text{LHS } \eqref{eq: main duality} =  \E_\mu[\I_A(Y_1) \I_{B_2 \cap \hat N}(Y_0) \hat P_{A^c}^{entr}(Y_0, B_1)]  = \text{RHS } \eqref{eq: main duality}, 
\]
where in the last equality we replaced $\hat N$ by $N$ using~\eqref{eq: N and N}.

It remain to establish \eqref{eq: dual alternation}, where the second equality follows from the first one by swapping $A$ and $A^c$. It suffices prove the first equality only on measurable subsets of $A^c$. For any such set $B$, by the definitions of $ \mu_{A}^{entr}$ and $\tilde \mu_{A}^{entr}$, we have
\[
\P_{\tilde \mu_{A}^{entr}}(Y^{\rangle A^c}_1 \in B) =\int_{A^c} \mu(dx) \int_{A \cap N}  \P_z(Y_1^{\rangle A^c} \in B)  \P_x(Y_1 \in dz). 
\]
Then, arguing as in the proof of \eqref{eq: main duality},
\begin{align*}
\P_{\tilde \mu_{A}^{entr}}(Y^{\rangle A^c}_1 \in B)&= \sum_{k=1}^\infty \P_\mu \big((Y_n)_{n=0}^{k+1} \in A^c \times  (A \cap N)^k \times (B  \cap N ) \big) \\
&= \sum_{k=1}^\infty \PP^{\hat Y}_\mu \big((y_n)_{n=0}^{k+1} \in (B  \cap \hat N ) \times  (A \cap \hat N)^k \times A^c \big) \\
&= \int_{B  \cap \hat N} \mu(dx) \int_{A \cap \hat N} \PP^{\hat Y}_z(\tau_{A^c}< \infty) \, \hat P(x, dz) \\
&= \int_{B  \cap \hat N} \mu(dx) \hat P(x, A) = \tilde \mu_{A^c}^{entr}(B).
\end{align*}

\end{proof}

\begin{proof}[{\bf Proof of Theorem~\ref{thm: invariant dual}}]
Let $P$ be the transition kernel of the chain $Y$. It is invariant with respect to the $\sigma$-finite measure $\mu$. Because $\XX$ is a Polish space, there exists a transition kernel $\hat P$ on $\XX$ that is dual to $P$ relative to $\mu$. Then there exists a dual chain $\hat Y$ with the transition kernel $\hat P$.

We already proved equality \eqref{eq: measures symm}, see \eqref{eq: mu^entr transf} above. We have
$N_A(Y) = \cap_{k=1}^\infty N_k$, hence  $\mu (N_A(Y)^c )= \mu( \cup_{k=1}^\infty N_k^c)=0$ by \eqref{eq: N_k co-null}, which we obtained from \eqref{eq: Y crosses} using no assumptions other than the invariance of $\mu$ for~$Y$. Similarly, $\mu (N_A(\hat Y)^c ) =0$ by \eqref{eq: dual assumptions}. Then the measure $\mu_{A^c}^{exit}$ is proper for the chain  $Y^{A^c \rangle}$ and it equals $\tilde \mu_{A^c}^{exit}$ restricted to $\FF_{A^c}$. Hence $\mu_{A^c}^{exit}$ is invariant for  $Y^{A^c \rangle}$ since so is $\tilde \mu_{A^c}^{exit}$ by Proposition~\ref{prop: duality}. By the same reasoning, $\mu_{A}^{entr}$ is invariant for $Y^{\rangle A}$.

Finally, we prove Kac's formula \eqref{eq: Kac MC}. Writing its r.h.s.\ using \eqref{eq: Kac proof}, we get
\begin{align*}
\int_A \E_x \! \Bigg [ \sum_{k=0}^{T_1^{\rangle A} - 1} \I(Y_k \in B ) \Bigg ] \mu_A^{entr}(dx)
&= \E_\mu \Bigg [ \sum_{k=1}^{T_C(Y)} \I(Y_0 \in A^c, Y_1 \in A, Y_k \in B) \Bigg ] \\
&= \sum_{k=1}^\infty \P_\mu(Y_0 \in A^c, Y_1 \in A, Y_k \in B, T_C(Y) \ge k)\\
&= \sum_{k=1}^\infty \P_\mu\Big(\hat Y_0 \in B, \tau_{A \times A^c} (\hat Y) = k -1 \Big) = \mu(B).
\end{align*}
\end{proof}

\section{Applications to random walks in $\R^d$} \label{Sec: Application to RWs}

In this section we apply the ideas developed in Sections~\ref{sec: stationarity MC} and~\ref{sec: duality} to random walks in $\R^d$. In particular, we answer our initial questions on stationarity properties of the chain of overshoots of a one-dimensional random walk over the zero level. 

Recall that $\ZZ$ denotes the minimal topologically closed subgroup of $(\R^d, +)$ that contains the topological support of the distribution of $X_1$. We assume throughout that $\ZZ$ has full dimension and $S_0 \in \ZZ$. We call $\ZZ$ the {\it state space} of the walk $S$. Denote by $\lambda$ the Haar measure on $\ZZ$ normalized such that $\lambda(Q)=1$, where $Q:=\{x \in \ZZ: 0 \le x < 1\}$ and we always mean that inequalities between points in $\R^d$ hold coordinate-wise. Clearly, $\lambda$ is invariant for the walk $S$ on $\XX=\ZZ$. 

We say that a Borel set $A \subset \ZZ$ is {\it massive} for the random walk $S$ if $\P_x(\tau_A(S)<\infty)=1$ for $\lambda$-a.e.\ $x \in \ZZ$. Since $-A$ is massive for $S$ if and only if $A$ is massive for $-S$, and the random walk $-S$ is dual to $S$ relative to the measure $\lambda$ (see~\cite[Eq.~(2.24)]{MijatovicVysotsky}), from Theorem~\ref{thm: invariant dual} we immediately obtain the following result.

\begin{theorem} \label{thm: RW general}
Assume that the sets $A$, $-A$, $A^c$, $-A^c$ are massive for a random walk $S$ on its state space $\ZZ$, where $\ZZ \subset \R^d$ and $d \ge 1$. Then the measures $\P(X_1 \in x - A^c) \lambda(dx)$ on $A$ and $\P(X_1 \in A - x) \lambda(dx)$ on $A^c$ are invariant for the entrance chain $S^{\rangle A}$ and exit chain $S^{A^c \rangle}$, respectively.
\end{theorem}

\begin{remark}
If $\ZZ=\Z^d$ with $d \ge 3$, $\E X_1 =0$ and $\E \|X_1\|^2 <\infty$, then the assumptions on $-A$ and $-A^c$ in Theorem~\ref{thm: RW general} are not required since by Uchiyama~\cite{Uchiyama}, a set is massive for such $S$ whenever it is massive for a simple random walk, which is self-dual. 
\end{remark}

For a particular example of $A$, consider the orthants in $\R^d$. We have the following result, which  we prove below after further comments. 

\begin{cor} \label{cor: orthant}
Put $\tau_\pm:=\tau_{\pm(0, \infty)^d}(S)$ and assume that $\P_0(\tau_\pm<\infty)=1$. Then the measures $\pi_+$ (defined in \eqref{eq: inv quadrant}) and
\[
\pi_-(dx):= (1 - \P(X_1 > x)) \lambda(dx), \qquad x \in \ZZ \cap (-\infty,0)^d,
\] 
are invariant for the chains of entrances of $S$ into $[0, \infty)^d$ and $(-\infty, 0)^d$, respectively. Moreover, for $d=1$, the measure $\pi$ (defined in \eqref{eq: pi}) is invariant for the chain of overshoots~$\mathcal O$.
\end{cor}


The assumptions of the corollary imply that every coordinate of $X_1$ has either zero mean or no expectation. In dimension one  $\tau_+$ and $\tau_-$ are the first strict ascending and descending ladder times of the random walk $S$ when $S_0=0$. Both quantities are finite a.s.\ if and only if $S$ {\it oscillates}, that is $\limsup S_n = -\liminf S_n = +\infty$ a.s.\ as $n \to \infty$. Moreover, $\tau_+$ and $\tau_-$ are finite a.s.\ if and only if
\[
\sum_{n=1}^\infty \frac1n \P_0(S_n >0)=\sum_{n=1}^\infty \frac1n \P_0(S_n <0) = +\infty;
\]
cf.~Feller~\cite[Theorems~XII.2.1 and~XII.7.2]{Feller}. This equivalence remains valid in dimension $d=2$; see Greenwood and Shaked~\cite[Corollary~3]{GreenwoodShaked}.

One can verify massiveness of a general set $A \subset \ZZ$ for a random walk $S$ using the following results. If $S$ is topologically recurrent, then any Borel set of positive measure $\lambda$ is massive for $S$, as follows (see Lemma~\ref{lem: induced2}) from ergodicity and recurrence of $\lambda$ for $S$ (see Lemma~\ref{lem: RW ergodic}). If $S$ is transient (i.e.\ not topologically recurrent), no set of finite measure can be massive. For walks on $\ZZ=\Z^d$ with $d \ge 3$  satisfying $\E X_1 =0$ and $\E \|X_1\|^2 <\infty$, there is a necessary and sufficient condition for massiveness of a set, called Wiener's test, stated in terms of capacity, by It\^o and McKean~\cite{ItoMcKean} and Uchiyama~\cite{Uchiyama}. Easily verifiable sufficient conditions for massiveness in $d=3$ are due to Doney~\cite{Doney1965}. For example, any straight ``line'' in $\Z^3$ is massive. We are not aware of any explicit results for random walks with a general distribution of increments apart from the partial results of Greenwood and Shaked~\cite{GreenwoodShaked} for convex cones with the apex at the origin. Based on the estimates of Green's function in Uchiyama~\cite[Section~8]{UchiyamaGreen}, it appears that such results whould be fully analogous to the ones for walks on $\ZZ=\Z^d$ if $\E X_1 =0$, $\E \|X_1\|^2 <\infty$, and the distribution of $X_1$ has density with respect to the Lebesgue measure. The case of heavy-tailed random walks on $\Z^d$, including transient walks in dimensions $d \in \{1,2\}$, is considered by Bendikov and Cygan~\cite{BendikovCygan1, BendikovCygan2}. 

\begin{proof}[{\bf Proof of Corollary~\ref{cor: orthant}}]
For the non-negative orthant $A=[0, \infty)^d$, we have
\[
\{X_1 \in x - A^c\} = \{X_1 \in (x - A)^c\} = \{X_1 \not \in x - A\} = \{X_1 \not \le x\}.
\]
Since the complement of each of the orthants $\pm(0, \infty)^d$ contains the other one, the result on $\pi_+$ follows from Theorem~\ref{thm: RW general} once we show that both orthants are massive. Equivalently, that $\tau_{x \pm(0, \infty)^d}(S)$ are finite $\P_0$-a.s.\ for every $x \in \ZZ$. We have $\tau_{x +(0, \infty)^d}(S) \le \tau_{x +(0, \infty)^d}(H)$, where $H$ is a random walk on $\ZZ$ defined by $H_n:=S_{\tau_n}$, where $\tau_0:=0$ and $\tau_n:=\inf\{k> \tau_{n-1}: S_k> S_{\tau_{n-1}}\}$ for $n \in \N$. Then every $\tau_{x +(0, \infty)^d}(H)$ is finite $\P_0$-a.s.\ since every coordinate of $H_n$ tends to $\infty$ as $n \to \infty$ by $H_1 >0$. Hence $\tau_{x +(0, \infty)^d}(S)$ is finite $\P_0$-a.s.\ and the same applies to $\tau_{x -(0, \infty)^d}(S)$.

The result on $\pi_-$ is analogous. For $d=1$, from the invariance of $\pi_+$ and $\pi_-$ it follows that $\pi$ is invariant for the sub-chain $(\mathcal O_{2n})_{n \ge 1}$. The invariance of $\pi$ for the full chain $\mathcal O$ follows from \eqref{eq: dual alternation}.
\end{proof}

Recall that the random walk $S$ is {\it topologically recurrent} $\P_0(S_n \in G \text{ i.o.})=1$ for every open neighbourhood $G$ of $0$. For such random walks, this equality is in fact true for every non-empty set $G \subset \ZZ$ that is open in the relative topology of $\ZZ$; see Revuz~\cite[Proposition~3.4]{Revuz}. 
Combined with the results of Chung and Fuchs~\cite[Theorems 1, 3 and~4]{ChungFuchs}, this implies that $S$ is topologically recurrent if and only if
\[
\limsup_{r \to 1-} \int_{[-a,a]^d} \frac{1}{\Re (1- r\E e^{i t \cdot X_1})}dt = \infty \quad \text{for all }a>0;
\]
the limit is always finite for $d \ge 3$. The limit commutes with the integral if $d=1$ (Ornstein~\cite[Theorem~4.1]{Ornstein}) or $\ZZ=\Z^d$ (Spitzer~\cite[Theorem~8.2]{Spitzer}). In particular, for $d=1$ this integral diverges when $\E X_1=0$, and it may also diverge for arbitrarily heavy-tailed $X_1$ (Shepp~\cite{Shepp}). In dimension $d=2$, $S$ is topologically recurrent on $\ZZ$ if $\E X_1=0$ and $\E \|X_1\|^2 <\infty$ (Chung and Lindvall~\cite{ChungLindvall}). For more general results on recurrence of random walks on locally compact Abelian metrizable groups, see Revuz~\cite[Chapters 3.3 and 3.4]{Revuz}.

We will present our uniqueness results after establishing the following lemma.

\begin{lemma} \label{lem: RW ergodic}
Let $S$ be a topologically recurrent random walk on $\R^d$. Then $\lambda$ is
the unique (up to multiplication by constant) locally finite Borel invariant measure of $S$ on $\ZZ$, and $\lambda$ is recurrent and ergodic for $S$.
\end{lemma}

\begin{proof}
The uniqueness of $\lambda$ is by Proposition~I.45 in Guivarc'h et al.~\cite{French}, which states that the right Haar measure on a locally compact Hausdorff topological group $G$ with a countable base is a unique invariant Radon Borel measure for any topologically recurrent right random walk on $G$ such that no proper closed subgroup of $G$ contains the support of the distribution of increments of the walk. 

To establish the ergodicity, we first note that the uniqueness of invariant measure $\lambda$ for $S$ implies it irreducibility for $S$. In fact, if there is a $\lambda$-non-trivial invariant set $A \in \mathcal{B}(\ZZ)$ of $S$, then the locally finite measure $\I_A \lambda$ is invariant for $S$, in contradiction with the uniqueness. From the topological recurrence of $S$ and Lemma~\ref{lem: cond rec}.\ref{cond: recur}  applied to any sequence of bounded open sets $B_n$ that cover $\ZZ$, we see that $\lambda$ is recurrent for $S$. Then $\lambda$ is ergodic for $S$ by Lemma~\ref{lem: cond rec}.\ref{cond: ergodic}. 
\end{proof}

\begin{theorem}
\label{thm: uniqueness}
Let $S$ be any topologically recurrent random walk on $\R^d$, and let $A \subset \ZZ$ be any $\lambda$-non-trivial Borel set with $\lambda(\partial A)=0$. Then $\lambda_A^{entr}$ is the unique (up to multiplication by constant) locally finite Borel measure on $A$ that is invariant for the chain $S^{\rangle A}$, and $\lambda_A^{entr}$ is recurrent and ergodic for $S^{\rangle A}$.  Similarly, $\lambda_{A^c}^{exit}$ is recurrent, ergodic, and uniquely invariant on $A^c$ for  $S^{A^c \rangle}$.
\end{theorem}
\begin{cor}
\label{cor: overshoots}
If a one-dimensional random walk $S$ is topologically recurrent, then the measures $\pi_+$, $\pi_-$, and $\pi$ are recurrent, ergodic, and iniquely invariant  
for the chains of overshoots $O$, $O^\downarrow$, and $\mathcal{O}$ (where $O_n^\downarrow:= \mathcal{O}_{2n - \I(S_0\ge 0)}$ for $n \ge 1$).
\end{cor}

Note that our upcoming paper~\cite{Vysotsky2024} gives a stronger uniqueness result under the minimal assumption that $S$ oscillates.

\begin{proof}[{\bf Proofs}]
The transition probability of $S$ is weak Feller by $\P_x(S_1 \in \cdot)=\P(x+ X_1 \in \cdot)$. Since $\tau_{\Cl(A)}(S)$ is finite $\P_\lambda$-a.e.\ by the ergodicity of $\lambda$ for $S$ (see Lemma~\ref{lem: induced2}), it follows that $\P_\lambda(S_{\tau_{\Cl(A)}(S)} \in \partial A)=0$ by $\lambda(\partial A)=0$. Hence $\P_x(S_{\tau_{\Cl(A)}(S)} \in \Int(A))=1$ for $\lambda$-a.e.\ $x \in \Int(A^c)$. Therefore, $\P_x (S_1 \in \Int(A))>0$ for some $x \in \Int(A^c)$ because $A$ is $\lambda$-non-trivial and  $\lambda(\partial A)=0$. Thus, the assumptions of  Theorem~\ref{thm: inducing bijection} are satisfied. Combined with Lemma~\ref{lem: RW ergodic}, this result implies Theorem~\ref{thm: uniqueness}.

The corollaries on $\pi_+$ and $\pi_-$ follow directly from Theorem~\ref{thm: uniqueness} by $\pi_+= \lambda_{[0, \infty)}^{entr}$ and $\pi_- = \lambda_{(- \infty,0)}^{entr} $. Furthermore, if $\pi$ is not ergodic, then by Lemma~\ref{lem: cond rec}.\ref{cond: ergodic} there is a $\pi$-non-trivial Borel set $A \subset \ZZ$ that is invariant for $\mathcal O$, i.e.\ $\P_x(\mathcal O_1 \in A) = \I_A(x)$ for $\pi$-a.e.\ $x \in \ZZ$. Then $A \cap [0, \infty)$ is a $\pi_+$-non-trivial invariant set for $O$ or $A \cap(-\infty, 0)$ is a $\pi_-$-non-trivial invariant set for $O^\downarrow$, in  contradiction with the ergodicity of $\pi_+$ and $\pi_-$. A similar argument yields the uniqueness of~$\pi$.
\end{proof}

Finally, let us comment on stability of the ``distribution'' of the entrance chain into $A$. This question makes a probabilistic sense only if the measure $\lambda_A^{entr}$ is finite and therefore can be normalized to be a probability. For example, this is true when $d=1$, $A = [0, \infty) $, $\E X_1 =0$ or when $S$ is topologically recurrent on $\ZZ$, $A$ is bounded, and $d \in \{1,2\}$. In the former case,  stability of the entrance chain  was studied in our paper~\cite{MijatovicVysotsky}. In the latter case, it is reasonable to restrict the attention to convex and compact sets $A$. These sets are intervals when $d=1$, considered in~\cite[Section~5.1]{MijatovicVysotsky}. It appears that convergence results in dimension $d=2$ can be obtained using exactly the same approach as in~\cite{MijatovicVysotsky}.


\section{The number of level-crossings for one-dimensional random walks} \label{Sec: L_n}

Throughout  this section we assume that the random walk $S$ is  one-dimensional. 

\subsection{Limit theorem} 
The main result of this section is the following theorem for $L_n$, the number of zero-level crossings of  $S$ by time~$n$, defined in~\eqref{eq: Ln def}.

\begin{theorem} \label{thm: level-crossings}
For any random walk $S$ such that $\E X_1 =0$ and $\sigma^2:=\E X_1^2 \in (0, \infty)$, we have
\begin{equation} \label{eq: L_n CLT}
\lim_{n \to \infty} \P_x \bigg( \frac{\sigma L_n}{ \E|X_1| \sqrt{n}} \le y \bigg) = 2\Phi (y)-1, \qquad x \in \ZZ, \, y \ge 0,
\end{equation}
where $\Phi$ denotes the distribution function of a standard normal random variable.
\end{theorem}

This weak convergence was first proved by Chung~\cite{Chung} under the additional assumption $\E X_1^3<\infty$. Maruyama~\cite[Theorem 3]{Maruyama} claimed this under $\E X_1^2 < \infty$ but it appears that his proof actually assumes  that $\E X_1^{2+\varepsilon} < \infty$ for some $\varepsilon >0$; indeed, the third equality in \cite[Eq.~(3.6)]{Maruyama} seems to rely on the argument used after \cite[Eq.~(3.1)]{Maruyama}.  In a more general setting, in the late 1950s I.I.\ Gikhman studied the number of
crossings of the boundary of a domain by a sequence of Markov chains converging to a diffusion process. His work was presented in two little known papers written in Ukrainian. According to Portenko~\cite[Section~9]{Portenko}, who surveyed these papers, Gikhman as well proved~\eqref{eq: L_n CLT} under the additional assumption $\E X_1^3<\infty$. In the early 1980s, A.N.~Borodin obtained limit theorems of more general type for additive functionals of consecutive steps of random walks; see~\cite[Chapter V]{BorodinIbragimov} and references therein. However, his method is limited by the assumption that the distribution of increments of the walk is either aperiodic integer-valued or has a square-integrable characteristic function, and hence absolutely continuous (by Kawata~\cite[Theorem~11.6.1]{Kawata}).

To prove Theorem \ref{thm: level-crossings}, we will combine the ergodicity of the chain of overshoots with a limit theorem for local times of random walks by Perkins~\cite{Perkins}. Specifically, we need the following result.

\begin{proposition} \label{prop: LLN}
Let $S$ be any random walk such that $\E X_1 =0$ and $\sigma^2:=\E X_1^2 \in (0, \infty)$. Then for \emph{every} $x \in \ZZ$,
\begin{equation} \label{eq: LLN overshoots}
\lim_{n \to \infty} \frac{1}{n} \sum_{k=1}^n |\mathcal{O}_k| =  \frac{1}{\E |X_1|}\int_{\ZZ} |y| \pi(dy) = \frac{\sigma^2}{2 \E |X_1|}, \quad \P_x \text{-a.s.}
\end{equation}
\end{proposition}

We stress that this does {\it not} follow directly from the ergodicity of $\mathcal O$ (stated in Corollary~\ref{cor: overshoots}) since Birkhoff's ergodic theorem implies convergence of the time averages only for $\pi$-a.e.\ $x \in \ZZ$ rather than for all $x$.

\begin{proof}[{\bf Proof of Theorem~\ref{thm: level-crossings}}]
Denote by $\ell_0$ the local time at $0$ at time $1$ of a standard Brownian motion. By L\'evy's theorem, $\ell_0$ has the same distribution as the absolute value of a standard normal random variable. Combining this result with Theorem~1.3  by Perkins~\cite{Perkins} and accounting for the $1/2$ in the definition of the Brownian local time in~\cite{Perkins}, we get
\begin{equation} \label{eq: Perkins}
\lim_{n \to \infty} \P_x \bigg (\frac{2}{\sigma \sqrt{n}} \sum_{k=1}^{L_n} |\mathcal{O}_k|  \le y \bigg ) = 2\Phi (y)-1, \qquad x=0, \, y \ge 0;
\end{equation}
since Perkins's definition of crossing times is slightly different from the one of ours, his result shall be applied to the random walk $-S/\sigma$. On the other hand, by Proposition~\ref{prop: LLN}, 
\begin{equation} \label{eq: LLN overshoots 2}
\lim_{n \to \infty} \frac{1}{L_n'} \sum_{k=1}^{L_n} |\mathcal{O}_k| = \frac{\sigma^2}{2 \E |X_1|}, \quad \P_x \text{-a.s.}, \qquad x \in \ZZ,
\end{equation}
where $L_n':=L_n+\I(L_n = 0)$ and we used the fact that $\P_x(\lim_{n \to \infty} L_n = \infty)=1$, which holds true since $S$ oscillates. Rewriting  equality \eqref{eq: Perkins} using the identity $\frac{1}{\sqrt n} = \frac{L_n'}{\sqrt n} \cdot \frac{1}{L_n'}$ and then combining it with \eqref{eq: LLN overshoots 2} yields the assertion of  Theorem~\ref{thm: level-crossings} for $x=0$ by Slutsky's theorem.

Furthermore, the results of Perkins actually imply (by Perkins~\cite{PerkinsPrivate}) that equality \eqref{eq: Perkins} remains valid, although this is not stated in~\cite[Theorem 1.3]{Perkins}, if we replace $x = 0$ by $x_n \in \ZZ$ for any sequence $(x_n)_{n \ge 1} \subset \ZZ$ such that $\lim_{n \to \infty} x_n/\sqrt n =0$. In particular, we can take $x_n \equiv x$ for an arbitrary $x \in \ZZ$, which yields Theorem~\ref{thm: level-crossings} in full by the above argument.

Let us explain in detail this extension of \eqref{eq: Perkins}. For $x=0$, Theorem 1.3 of  Perkins~\cite{Perkins}  is an immediate corollary to his Lemma 3.2 and Corollary 2.2. Our extension of \eqref{eq: Perkins}  follows in exactly the same way if in~\cite[Lemma~3.2]{Perkins} we let
 $x$ be the {\it nearstandard} point in $^*\R$ (the field of {\it nonstandard real numbers}) that corresponds to the sequence $(x_n)_{n \ge 1} $, in which case $st(x) = \text{\textdegree} x = 0$, i.e.\ the {\it standard part} of $x$ is $0$. We referred to Cutland~\cite{Nonstandard} to digest the  unusual notation and concepts of nonstandard analysis, which were used in~\cite{Perkins} with no explanation.
\end{proof}

\begin{proof}[{\bf Proof of Proposition~\ref{prop: LLN}}]


Denote $h:=\inf\{z \in \ZZ: z>0\}$; then either $\ZZ=h\Z$ if $h>0$ or $\ZZ=\R$ if $h=0$.
One can easily check that for $\pi_+$ (defined in \eqref{eq: inv quadrant}),
\begin{equation} \label{eq: E 1}
\int_{\ZZ \cap [0, \infty)} y \pi_+(dy) =  \int_h^\infty (y -h/2) \P(X_1 >y) dy =\int_0^\infty (y -h/2) \P(X_1 >y) dy 
\end{equation}
and, similarly, 
\begin{equation} \label{eq: E 2}
-\int_{\ZZ \cap (-\infty, 0)} y \pi_-(dy) = \int_0^\infty (y +h/2) \P(-X_1 >y) dy.
\end{equation}
Using that $\E X_1 = 0$ and integrating the above equality by parts, we find that the {\it probability} measure $\pi/ \E |X_1|$ has the first absolute moment $\sigma^2/(2 \E|X_1|)$. Therefore, by Birkhoff's ergodic theorem (see the Appendix) and ergodicity of the chain of overshoots $\mathcal O$ asserted in Corollary~\ref{cor: overshoots}, the convergence in \eqref{eq: LLN overshoots} holds true for $\pi$-a.e.\ $x \in \ZZ$. We need to prove this for every $x \in \ZZ$. 

Denote by $\supp \pi$ the topological support of $\pi$ and by $N$ the set of points $x \in \supp \pi$ that satisfy \eqref{eq: LLN overshoots}. We clearly have $N = \supp \pi$ in the lattice case $h>0$, where $\ZZ$ is discrete. In the non-lattice case $h=0$, so far we only have that $N$ is dense in $\supp \pi$. This is because $N$ has full measure $\pi$, hence $N$ has full Lebesgue measure $\lambda|_{\supp \pi}$, as readily seen from definition~\eqref{eq: pi} of $\pi$. In order to prove \eqref{eq: LLN overshoots}, we need to show that $N = \supp \pi$, since the chain $\mathcal{O}$ hits the support of $\pi$ (which is a closed interval, possibly infinite) at the first step regardless of the starting point. Our argument goes as follows.

Consider the random walk $S':= (S_n')_{n \ge 0}$, where $S_n'= X_1+ \ldots + X_n$ for $n \ge 1$,  starting at $S_0':=0$. Then $\P_x(S \in \cdot) = \P\big((x+S_0', x+ S_1', \ldots) \in \cdot\big)$. For real $y_1, y_2$, define the functions
\[
g(y_1, y_2):= \I(y_1<0, y_2 \ge 0 \text{ or } y_1 \ge 0, y_2< 0), \quad f(y_1, y_2) := |y_2| g(y_1, y_2).
\]
We claim that for any $x \in \supp \pi$ and $\varepsilon \in (0,1)$, there exists a $ y \in N$ such that 
\begin{equation} \label{eq: approximation by N}
\limsup_{n \to \infty} \left | \frac{\sum_{k=1}^n f(y + S_{k-1}', y+ S_k')}{\sum_{k=1}^n g(y + S_{k-1}', y+ S_k')} -  \frac{\sum_{k=1}^n f(x + S_{k-1}', x+ S_k')}{\sum_{k=1}^n g(x + S_{k-1}', x+ S_k')} \right | \le \varepsilon, \quad \P \text{-a.s.}
\end{equation}
This will imply that $x \in N$ and hence prove Proposition~\ref{prop: LLN}, since 
\begin{equation*} 
\P \left( \lim_{n \to \infty}  \frac{\sum_{k=1}^n f(y + S_{k-1}', y + S_k')}{\sum_{k=1}^n g(y + S_{k-1}', y + S_k')} = \frac{\sigma^2}{2 \E |X_1|} \right) = \P_y \left( \lim_{n \to \infty}  \frac{1}{L_n'} \sum_{k=1}^{L_n}  |\mathcal{O}_k| = \frac{\sigma^2}{2 \E |X_1|} \right) = 1,
\end{equation*}
where $L_n'=L_n+\I(L_n = 0)$ and  the last equality holds by definition of the set $N$ and the fact that $\P_y(\lim_{n \to \infty} L_n = \infty)=1$, which is true because $S$ oscillates. Thus, it remains to prove inequality \eqref{eq: approximation by N}.

From the identity $\frac{a_1}{b_1} - \frac{a_2}{b_2}= \frac{a_1}{b_1} \bigl (1 -  \frac{a_2}{a_1} \cdot \frac{b_1}{b_2}  \bigr)$ for $a_1, a_2, b_1, b_2 >0$ and the fact that $\big |  1 -\frac{a_2}{a_1} \big | + \big | 1 -\frac{b_1}{b_2} \big | \le \frac12$  implies $ \big |  1 -\frac{a_2}{a_1} \cdot \frac{b_1}{b_2} \big | <1$,  we see that \eqref{eq: approximation by N} will follow if we show that for any $x \in \supp \pi$ and $\varepsilon \in (0, 1)$, there exists a $ y \in N$ such that $\P$-a.s.,
\begin{equation} \label{eq: bound}
\limsup_{n \to \infty} \left [ \left | \frac{\sum_{k=1}^n f(x+S_{k-1}', x + S_k')}{\sum_{k=1}^n f(y+S_{k-1}', y + S_k')} -1 \right| + \left | \frac{\sum_{k=1}^n g(x+S_{k-1}', x + S_k')}{\sum_{k=1}^n g(y+S_{k-1}', y + S_k')} -1 \right| \right ] \le \frac{ \varepsilon \E|X_1|}{\sigma^2}.
\end{equation}

For any $\delta >0$, integer $k \ge 1$, and any $y \in N$ such that $|x- y| \le \delta$, we have
\[
|g(x + S_{k-1}', x+ S_k') - g(y + S_{k-1}', y+ S_k')| \le \I(|y+ S_{k-1}'| \le \delta \text{ or } |y+ S_k'| \le \delta)
\]
and
\begin{multline*}
|f(x + S_{k-1}', x+ S_k') - f(y + S_{k-1}', y+ S_k')| \\
\le \delta g(y + S_{k-1}', y + S_k') + (|y + S_k'| + \delta) \I(|y+ S_{k-1}'| \le \delta \text{ or } |y+ S_k'| \le \delta).
\end{multline*}
This gives
\begin{equation} \label{eq: estimate g}
\left | \frac{\sum_{k=1}^n g(x+S_{k-1}', x + S_k')}{\sum_{k=1}^n g(y+S_{k-1}', y + S_k')} -1 \right| \\
\le \frac{\sum_{k=1}^n \bigl[ \I(|y+ S_{k-1}'| \le \delta) +  \I( |y+ S_k'| \le \delta) \bigr ]}{\sum_{k=1}^n g(y+S_{k-1}', y + S_k')}
\end{equation}
and
\begin{multline} \label{eq: estimate f}
\left | \frac{\sum_{k=1}^n f(x+S_{k-1}', x + S_k')}{\sum_{k=1}^n f(y+S_{k-1}', y + S_k')} -1 \right| \\
\le \frac{\sum_{k=1}^n \bigl[ \delta g(y + S_{k-1}', y + S_k') + (|X_k| + 2\delta) \I(|y+ S_{k-1}'| \le \delta) + 2 \delta \I( |y+ S_k'| \le \delta) \bigr ]}{\sum_{k=1}^n f(y+S_{k-1}', y + S_k')}.
\end{multline}

By Lemma~\ref{lem: RW ergodic}, the Lebesgue measure $\lambda$ is recurrent and ergodic for $S$ since $S$ is topologically recurrent by $\E X_1=0$. Then the measure preserving one-sided shift $\theta$ on $(\R^{\N_0}, \mathcal{B}(\R^{\N_0}), \PP_{\lambda}^S)$ is conservative by Lemma~\ref{lem: cond rec}.\ref{cond: conserv}). Therefore, we can apply Hopf's ratio ergodic theorem (see the Appendix) to the ratios on the r.h.s.'s of \eqref{eq: estimate g} and \eqref{eq: estimate f}. Let us explain in details, say, why 
\begin{equation}\label{eq:Hopfpart}
\P \left (\lim_{n\to \infty}\frac{\sum_{k=1}^n  g(y + S_{k-1}', y + S_k') }{\sum_{k=1}^n f(y+S_{k-1}', y + S_k')} = \frac{\E|X_1|}{\sigma^2/2}\right) =1, \qquad \lambda\text{-a.e. } y.
\end{equation}

Indeed, consider the functions on $\R^{\N_0}$ defined by $G(z):=g(z_0, z_1)$ and $F(z):=f(z_0,z_1)$ for $z=(z_0, z_1, \ldots) \in \R^{\N_0}$. Both functions are non-negative, non-zero, and $\PP_{\lambda}^S$-integrable by
\[
\EE_{\lambda}^S  G= \int_\R \E_{z_0} g(S_0, S_1)  \lambda(d z_0) 
=  \int_{-\infty}^0 \P(z_0+X_1 \ge 0) dz_0 + \int_0^\infty \P(z_0+X_1 <0) d z_0  = \E|X_1|
\]
and 
\begin{align*}
\EE_{\lambda}^S F&= \int_\R \E_{z_0} [|S_1| g(S_0, S_1)]  \lambda(d z_0) \\
&=  \int_{-\infty}^0 \E [(z_0 + X_1) \I (z_0+X_1 \ge 0)] dz_0 - \int_0^\infty \E [(z_0 + X_1) \I (z_0+X_1 < 0)] d z_0 \\
&=  \int_0^\infty \E [(| X_1| -z_0) \I (|X_1|  > z_0)] d z_0 = \E |X_1|^2/2,
\end{align*}
where the last equality follows from Fubini's theorem. Finally, we have
\begin{multline*}
\PP_{\lambda}^S \left (\limsup_{n\to \infty} \bigg| \frac{\sum_{k=0}^{n-1} G\circ \theta^k}{\sum_{k=0}^{n-1} F \circ \theta^k}  - \frac{\EE_{\lambda}^S G}{\EE_{\lambda}^S F}  \bigg| \neq 0  \right) \\
=
\int_\R \P \left (\limsup_{n\to \infty} \bigg| \frac{\sum_{k=1}^n  g(y + S_{k-1}', y + S_k') }{\sum_{k=1}^n f(y+S_{k-1}', y + S_k')} - \frac{\E|X_1|}{\sigma^2/2}\bigg| \neq 0  \right) \lambda(dy),
\end{multline*}
hence equality~\eqref{eq:Hopfpart} follows from Hopf's ratio ergodic theorem.

Similarly to~\eqref{eq:Hopfpart}, for every $\delta >0$, for $\lambda$-a.e.\ $y$ the sum of the r.h.s.'s of~\eqref{eq: estimate g} and \eqref{eq: estimate f} converges $\P$-a.s.\ as $n \to \infty$ to 
\begin{equation*} 
c(\delta):=\frac{\delta \E |X_1| + 2 \delta (\E |X_1| + 2\delta) + 4 \delta^2 }{\sigma^2/2} + \frac{4 \delta}{\E |X_1|}.
\end{equation*}
Denote by $N_\delta$ the set of $y$ where this $\P$-a.s.\ convergence holds true. Choose a $\delta>0$ such that $c(\delta)<\varepsilon \E|X_1|/ \sigma^2$. The Borel set $N \cap N_\delta$ has full measure $\lambda|_{\supp \pi}$ and hence is dense in $\supp \pi$. Therefore we can pick a $y \in N \cap N_\delta$ that satisfies $|x-y| \le \delta$. Then inequality~\eqref{eq: bound} follows from \eqref{eq: estimate g} and \eqref{eq: estimate f}, as required. 
\end{proof}

\subsection{Stationarity of level-crossings}
Define the first up-crossing time of zero   
\[
 T:= \inf\{k\ge 1: S_{k-1} < 0, S_k \ge 0 \},
\]
and the {\it numbers of up and down-crossings} of an arbitrary level $a \in \ZZ$ by time $n \ge 1$: 
\[
L_n^\uparrow(a):=\sum_{i=0}^{n-1} \I(S_i <a, S_{i+1} \ge a),  \qquad L_n^\downarrow(a):=\sum_{i=0}^{n-1} \I(S_i \ge a, S_{i+1} < a).
\]
Recall that the measures $\pi_+$ and $\pi_-$, defined in~\eqref{eq: inv quadrant}, have the total mass $\E |X_1|/2$ each when $\E X_1=0$. Put $\pi_\pm ':= 2 \pi_\pm / \E |X_1|$. We have the following rather surprising result.

\begin{proposition} \label{prop: E up-crossings}
For any non-degenerate random walk $S$ that satisfies $\E X_1=0$ and any $a \in \ZZ$, we have 
\[
\E_{\pi_+'} L_T^\uparrow(a) = \E_{\pi_-'} L_T^\uparrow(a) = \E_{\pi_+'} L_T^\downarrow(a) = \E_{\pi_-'} L_T^\downarrow(a) =1.
\]
\end{proposition}

Thus, the expected number of up-crossings of a level by the  time $T$ does not depend on the level if $S$ is started under $\pi_+'$ or $\pi_-'$ (i.e.\ at stationarity of either chain $O$ or $O^\downarrow$), and therefore equals~$1$ since $L_T^\uparrow(0)=1$ by the definition of $T$. In the particular case when $S$ is a symmetric simple random walk, this is a well-known fact  (see e.g.~Feller~\cite[Section XII.2, Example~b]{Feller}) since here $\pi_+'=\delta_0$ and  $L_T^\uparrow(a) + L_T^\downarrow(a)$ is the local time of the walk at level $a$. 

\begin{proof}
We use Kac' formula~\eqref{eq: Kac} for the measure-preserving  shift $\theta$ on $(\ZZ^{\N_0}, \mathcal{B}(\ZZ^{\N_0}), \PP_{\lambda}^S)$ and $A=\{x \in \ZZ^{\N_0}: x_0 < 0, x_1 \ge 0 \}$. For the up-crossings, take $B=\{x \in \ZZ^{\N_0}: x_0 < a, x_1 \ge a \}$ and use that $T$ is the first entrance time of $S$ into $[0, \infty)$. By the same computation as in~\eqref{eq: Kac final}, this gives $\pi_+(\ZZ \cap [0, \infty)) = \E_{\pi_+} L_T^\uparrow(a) $. Similarly, take $A=\{x \in \ZZ^{\N_0}: x_0 \ge 0, x_1 < 0 \}$ to get $\pi_-(\ZZ \cap (-\infty,0)) = \E_{\pi_-} L_T^\uparrow(a) $. For the down-crossings, consider $B=\{x \in \ZZ^{\N_0}: x_0 \ge a, x_1 < a \}$.
\end{proof}

\section*{Acknowledgements}
We thank Vadim Kaimanovich for providing a reference to his extremely useful paper. 
AM is supported by EPSRC grants EP/V009478/1 and EP/W006227/1.
The work of VV was supported in part by Dr Perry James (Jim) Browne Research Centre.

\appendix
\section{Induced transformations in infinite ergodic theory} 

Here we present some relevant basic results on inducing for measure preserving transformations of infinite measure spaces;  see Aaronson~\cite[Chapter 1]{Aaronson} for an introduction. To our surprise, we failed to find straightforward references to the results needed.

Let $T$ be a measure preserving transformation of a measure space $(X, \mathcal{F},m)$. For any set $A \in \mathcal{F}$, consider the {\it first hitting time} $\tau_A$ of $A$ and the {\it induced mapping} $T_A$ defined by
\begin{equation} \label{eq: tau phi def}
\tau_A(x):=\inf\{n \ge 1: T^n x \in A\}, \,  x \in X \quad \text{and} \quad T_A(x):=T^{\tau_A(x)} x, \, x \in \{\tau_A < \infty\}.
\end{equation}
These mappings are measurable. Note that this notation differs from the one in Section~\ref{sec: setup MCs}. 

We say that a set $A \in \mathcal F$ is {\it recurrent} for $T$ if $\tau_A$ is finite $m$-a.e.\ on $A$, that is $A \subset \cup_{k \ge 1} T^{-k} A \Mod{m}$, where$\Mod{m}$ means true possibly except for a $m$-zero set. For such $A$, it follows by simple induction from the invariance of $m$ that 
\begin{equation} 
\label{eq: io}
m(A \cap \{\tau_A = \infty\})=0 \quad \Longrightarrow \quad m \Big(A \cap \Big\{ \sum_{n=1}^\infty \I(T^n \in A) < \infty \Big \} \Big)=0,
\end{equation}
hence all iterations of the mapping  $T_A$ are defined $m$-a.e.\ on~$A$. The transformation $T$ is called {\it ergodic} if its invariant $\sigma$-algebra $\mathcal{I}_T :=\{A \in \mathcal{F}: T^{-1} A = A \Mod{m} \} $ is $m$-trivial, i.e.\ for every $A \in \mathcal{I}_T$ either $m(A)=0$ or $m(A^c)=0$.

The following statement essentially is~\cite[Proposition~1.5.3]{Aaronson} (which is stated under slightly different assumptions but its proof works unchanged).

\begin{lemma} \label{lem: induced'}
Let $T$ be a measure preserving transformation of a measure space $(X, \mathcal{F}, m)$, and $A \in \mathcal{F}$ be any set recurrent for $T$ such that $m(A)>0$. We have $m_A \circ T_A^{-1} \le m_A$ on $(A,\mathcal{F}_A)$, and if $m(A)<\infty$, then the induced mapping $T_A$ is a  measure preserving transformation of the induced space $(A, \mathcal{F}_A, {m}_A)$.
\end{lemma}

To relax the condition $m(A)<\infty$, we need additional assumptions. 
We say that $T$ is {\it conservative} if every measurable subset of $X$ is recurrent for $T$. 
\begin{lemma} \label{lem: induced2}
Let $T$ be a measure preserving conservative transformation of a $\sigma$-finite measure space $(X, \mathcal{F}, m)$, and $A \in \mathcal{F}$ be any set with $m(A)>0$. Then $T_A$ is a measure preserving conservative transformation of the induced space $(A, \mathcal{F}_A, m_A)$. Moreover, if $T$ is ergodic, then $T_A$ is ergodic and $X= \cup_{k \ge 1} T^{-k} A \Mod{m}$.
\end{lemma}
The first statement here is in Corollary 1.1 by P\`ene and Thomine~\cite{PeneThomine} (who formally assumed that $X$ is a Polish space but actually never used this in the proof of their Proposition~0.1). The second statement is in~\cite[Propositions~1.2.2 and~1.5.2]{Aaronson}. 

We now present conditions for conservativity. 

\begin{lemma} \label{lem: criteria}
Let $T$ be a measure preserving transformation of a measure space $(X, \mathcal{F}, m)$. Then $T$ is conservative and $m$ is $\sigma$-finite if and only if there exists a sequence of sets $\{ A_k\}_{k \ge 1} \subset \mathcal F$, all of finite measure and  recurrent for $T$, such that $X= \cup_{k \ge 1} A_k \Mod{m}$. In particular, the sufficient condition is satisfied when $X = \cup_{k \ge 1} T^{-k} A \Mod{m}$ for some measurable set $A$ of finite measure. 
\end{lemma}

The second assertion is known as Maharam's  recurrence theorem; see~\cite[Theorem~1.1.7]{Aaronson}. Its particular case is Poincar\'e's recurrence theorem that $T$ is conservative when $m$ is finite.  

\begin{proof}
The direct implication in the first assertion is trivial. For the reverse one,  assume that there is a set $A \in \mathcal F$ of positive measure that is not recurrent for $T$. Then so is $A': = A \setminus \cup_{n=1}^\infty T^{-n} A$. Pick a  $k \ge 1$ such that $m(A_k \cap A')>0$. By Lemma~\ref{lem: induced'}, the induced mapping $T_{A_k}$ is measure preserving on the induced space $(A_k, \mathcal F_{A_k}, m_{A_k})$ of finite measure. This mapping is conservative by Poincar{\'e}'s recurrence theorem, hence $A_k \cap A'$ is a recurrent set for $T_{A_k}$, hence it is recurrent for $T$, which is a contradiction. For the second assertion, take $A_k:=T^{-k} A$ for $k \ge 1$.
\end{proof}


The next result, known as Kac's formula, concerns reversing the inducing.
\begin{lemma} \label{lem: induced3}
Let $T$ be a conservative measure preserving transformation of a $\sigma$-finite measure space $(X, \mathcal{F}, m)$, and let $A \in \mathcal{F}$ be any set such that $X = \cup_{k \ge 1} T^{-k} A \Mod{m}$. Then 
\begin{equation} \label{eq: Kac}
m(B)= \int_A \Bigg [ \sum_{k=0}^{\tau_A(x) - 1} \I(T^k x \in B ) \Bigg ]\! m_A(dx), \qquad B \in \mathcal{F}.
\end{equation}
\end{lemma}

\begin{proof}
Denote the r.h.s.\ of \eqref{eq: Kac} by $\mu(B)$. By monotonicity and $\sigma$-finiteness of $m$, it suffices to check equality $m = \mu$ only on sets of finite measure $m$. For any measurable $B \subset X$,  
\begin{align} \label{eq: Kac proof 1}
\mu(B) =& \int_A  \Bigg [ \sum_{n=1}^\infty \I(\tau_A(x) = n) \times \sum_{k=0}^{\tau_A(x) - 1} \I(T^k x \in B) \Bigg ] \! m(dx) \notag \\
=& \int_A \left [ \sum_{n=1}^\infty \sum_{k=0}^{n - 1} \I(T^k x \in B, \tau_A(x) = n ) \right ] \! m(dx) \notag \\ 
=& \sum_{k=0}^\infty m (A \cap T^{-k} B \cap \{\tau_A > k\}),
\end{align}
and therefore, assuming that $m(B)<\infty$, we get
\[
\mu(B) = \sum_{k=0}^\infty m (A \cap T^{-k} B \setminus \cup_{n=1}^k T^{-n} A) = m(A \cap B) + \sum_{k=1}^\infty m(A \cap T^{-1} B_{k-1}'),
\]
where $B_k':= T^{-k} B \setminus \cup_{n=0}^k T^{-n} A$ for $k \ge 0$. The set $T^{-1} B_k'$ has finite measure and it is a disjoint union of $A \cap T^{-1} B_k'$ and $B_{k+1}'$, hence  $m( A \cap T^{-1} B_k') = m(B_k') - m(B_{k+1}')$. Then the sequence $m(B_k')$ is decreasing, and 
\begin{equation} \label{eq: lifted =}
\mu(B)= m(A \cap B) + m(B_0') - \lim_{k \to \infty} m(B_k') = m(B) -  \lim_{k \to \infty} m(B_k').
\end{equation}
It remains to show that the limit in the above formula is zero. 

For any integer $N \ge 1$, denote $B^{(N)}:=B \cap (\cup_{n=1}^N T^{-n} A)$. Notice that for any $k \ge N$, we have $\{ \tau_{B^{(N)}} \le k-N \} \subset \{ \tau_A \le k \}$, hence
\[
T^{-k} (B^{(N)}) \setminus \cup_{n=0}^k T^{-n} A \subset \{ k-N <\tau_{B^{(N)}} \le k\}, \qquad k \ge N.
\]
Then for $k \ge N$,
\begin{align*}
m(B_k')&=m \bigl(T^{-k} (B \setminus \cup_{n=1}^N T^{-n} A) \setminus \cup_{n=1}^k T^{-n} A \bigr) +m\bigl(T^{-k} (B^{(N)}) \setminus \cup_{n=1}^k T^{-n} A \bigr) \\
&\le m \bigl(T^{-k} (B \setminus \cup_{n=1}^N T^{-n} A) \bigr) + N \sup_{n > k-N} m(\tau_{B^{(N)}} = n) \\
&= m (B \setminus \cup_{n=1}^N T^{-n} A ) + N \sup_{n > k-N} m (T^{-n}(B^{(N)}) \setminus \cup_{i=1}^{n-1} T^{-i} B^{(N)}).
\end{align*}
The first term in the last line can be made as small as necessary by choosing $N$ to be large enough, and the second term vanishes as $k \to \infty$ for any fixed $N$ by Remark to Proposition~1.5.3 in~\cite{Aaronson}.
\end{proof}

The next result concerns extending an arbitrary invariant measure of the induced transformation to the invariant measure of the full one. 

\begin{lemma} \label{lem: Kac 2}
Let $T$ be a measurable transformation of a measurable space $(X, \mathcal{F})$, and $A \in \mathcal F$ be any non-empty set. Assume that $m_0$ is a measure on $(A, \mathcal F_A)$ such that $m_0(A \setminus \cup_{k \ge 1} T^{-k} A )=0$ and $m_0$ is invariant for $T_A$. Then the measure $m$ on $(X, \mathcal{F})$, defined by~\eqref{eq: Kac} with $m_0$ substituted for $m_A$, is  invariant for $T$ and satisfies $m_A=m_0$ and $X = \cup_{k \ge 1} T^{-k} A \Mod{m}$. If $m_0$ is $\sigma$-finite, then so is $m$. If $T_A$ is conservative on $(A, \mathcal F_A, m_0)$, then $T$ is conservative on $(X, \mathcal F, m)$. If $T_A$ is conservative and ergodic, then so is $T$. 
\end{lemma}
\begin{proof}
The equality $m_A=m_0$ is trivial. Invariance of $m$ is standard; for example, see the proof of~\cite[Proposition~1.5.7]{Aaronson}. For any integer $k \ge 0$, 
\[
m_0(A \cap \{\tau_A >k\} \cap \{\tau_A \circ T^k=\infty \}) = m_0(A \cap \{\tau_A =\infty \}) =0,
\]
hence $ m(\tau_A =\infty)=0$ as claimed. If $m_0$ is $\sigma$-finite, we have $A=\cup_{n=1}^\infty A_n$ for some measurable sets $A_n$ of finite measure  $m_0$. Then 
\[
X=\cup_{k\ge 1} T^{-k} A= \cup_{k,n \ge 1} T^{-k} A_n \Mod{m},
\]
where $ m(T^{-k} A_n)=m(A_n)=m_0(A_n)<\infty$, hence $m$ is $\sigma$-finite. 

If $T_A$ is conservative, then for any $B \in \mathcal F$, 
\begin{align*}
m(B\cap \{\tau_B=\infty\})&= m_0\big(A \cap\{\tau_B < \tau_A\} \cap \{\tau_B \circ T^k=\infty, \forall k \ge \tau_B\}\big) \\
&\le m_0\big(A \cap\{\tau_B < \infty\} \cap \{\tau_B \circ T_A^k=\infty, \forall k \ge 1\}\big) =0
\end{align*}
by recurrence of the set $A\cap \{\tau_B<\infty\}$ for $T_A$. This means that $B$ is recurrent for $T$, and thus $T$ is conservative. The assertion on ergodicity is in~\cite[Proposition~1.5.2]{Aaronson}.
\end{proof}

Finally, we recall the following classical result; see Zweim\"uller~\cite{Zweimuller}.
\begin{Hopf}
Let $T$ be a conservative ergodic measure preserving transformation of a $\sigma$-finite measure space $(X, \mathcal{F}, m)$. Then for any functions $f, g \in L^1(X, \mathcal{F}, m)$ with non-zero $g \ge 0$, 
\begin{equation*} 
\lim_{n \to \infty} \frac{\sum_{k=0}^{n-1} f \circ T^k}{\sum_{k=0}^{n-1} g \circ T^k}  = \frac{\int_X f d m}{\int_X g d m}, \quad m\text{-a.e.} 
\end{equation*}
\end{Hopf}

A particular case of this statement for probability measures is Birkhoff's ergodic theorem, which is obtained by taking $g=1$ when $m(X)=1$.

\bibliographystyle{plain}
\bibliography{overshoot}

\end{document}